\documentclass{amsart}

\usepackage{kotex}
\usepackage{amsmath}
\usepackage{amsfonts}
\usepackage{amscd}

\usepackage{amssymb}
\usepackage{graphicx}
\usepackage{caption}
\usepackage{subcaption}
\usepackage{dsfont}
\usepackage{color}
\usepackage{hyperref}
\usepackage{bbm}
\usepackage{a4wide}
\usepackage{sseq}
\usepackage{tikz-cd}
\usepackage[abs]{overpic} 

\newtheorem{theorem}{Theorem}[section]
\newtheorem{lemma}[theorem]{Lemma}
\newtheorem{proposition}[theorem]{Proposition}
\newtheorem{corollary}[theorem]{Corollary}

\theoremstyle{definition}
\newtheorem*{definition*}{Definition}
\newtheorem{definition}[theorem]{Definition}

\theoremstyle{remark}
\newtheorem*{remark*}{Remark}

\numberwithin{equation}{section}

\begin{document}

\title[DQ via Toeplitz Operators on GQ in Real Polarizations]{Deformation Quantization via Toeplitz Operators on Geometric Quantization in Real Polarizations}

\author{NaiChung Conan Leung, AND YuTung Yau}

\thanks{}

\maketitle

\begin{abstract}
	In this paper, we study quantization on a compact integral symplectic manifold $X$ with transversal real polarizations. In the case of complex polarizations, namely $X$ is K\"ahler equipped with transversal complex polarizations $T^{1, 0}X, T^{0, 1}X$, geometric quantization gives $H^0(X, L^{\otimes k})$'s. They are acted upon by $\mathcal{C}^\infty(X, \mathbb{C})$ via Toeplitz operators as $\hbar = \tfrac{1}{k} \to 0^+$, determining a deformation quantization $(\mathcal{C}^\infty(X, \mathbb{C})[[\hbar]], \star)$ of $X$.\par
	We investigate the real analogue to these, comparing deformation quantization, geometric quantization and Berezin-Toeplitz quantization. The techniques used are different from the complex case as distributional sections supported on Bohr-Sommerfeld fibres are involved.\par
	By switching the roles of the two real polarizations, we obtain Fourier-type transforms for both deformation quantization and geometric quantization, and they are compatible asymptotically as $\hbar \to 0^+$. We also show that the asymptotic expansion of traces of Toeplitz operators realizes a trace map on deformation quantization.
\end{abstract}

\address{}
 \email {}

\section{Introduction}
\label{Section 1}
The term `quantization' originates from physics, referring to the process of building quantum mechanics from classical mechanics. Consider the simplest example $X = T^*V$ with the canonical symplectic form $\omega = \sum_{i=1}^n dx^i \wedge dy^i$, where $(x^1, ..., x^n)$ are the standard coordinates on $V = \mathbb{R}^n$ and $(y^1, ..., y^n)$ are the dual coordinates on $V^*$. Via the action $x^i \mapsto x^i \cdot$, $y^j \mapsto \tfrac{\hbar}{2\pi \sqrt{-1}} \partial_{x^j}$ on $\mathcal{H} = L^2(V)$, there induces a non-commutative product $\star$ on $\mathcal{C}^\infty(X, \mathbb{C})[[\hbar]]$ with $\hbar^{-1}(f \star g - g \star f) = \tfrac{\sqrt{-1}}{2\pi}\{f, g\}$.\par
In the general case, the construction of the quantum Hilbert space $\mathcal{H}$ requires the condition $[\omega] = c_1(L) \in H^2(X, \mathbb{Z})$ for a Hermitian line bundle $L$ and a polarization of $(X, \omega)$. When $X$ is compact K\"ahler, there are transversal complex polarizations $T^{1, 0}X, T^{0, 1}X$. Geometric quantization gives the quantum Hilbert space $\mathcal{H}^\hbar = H^0(X, L^{\otimes k})$ in polarization $T^{1, 0}X$, where $\hbar = \tfrac{1}{k}$ and $k \in \mathbb{N}$. In \cite{S2000}, Schlichenmaier shows that the \emph{Toeplitz operators} $Q^\hbar: \mathcal{C}^\infty(X, \mathbb{C}) \to \operatorname{End}_\mathbb{C} \mathcal{H}^\hbar$ completely determine a deformation quantization $(\mathcal{C}^\infty(X, \mathbb{C})[[\hbar]], \star^{\operatorname{BT}})$ known as the \emph{Berezin-Toeplitz star product} so that it acts on $\mathcal{H}^\hbar$ asymptotically as $\hbar \mapsto 0^+$.\par
In this paper, we study the real analogue to Toeplitz operators, assuming $X$ is compact and is equipped with transversal real polarizations with compact leaves. Essentially, up to finite covers, $X = (V \oplus V^*) / (\Lambda \oplus \Lambda')$ with standard symplectic form, where $V \cong \mathbb{R}^n$ and $\Lambda, \Lambda'$ are lattices in $V, V^*$ respectively (c.f. Appendix \ref{Appendix A}). For simplicity, throughout of the rest of this paper except Section \ref{Section 2}, let $X = T \times \check{T}$, where $T = \check{T} = \mathbb{R}^n/\mathbb{Z}^n$, with the canonical symplectic form $\omega$ on $X$ viewed as the dual torus bundle over $T$. Define real polarizations $\mathcal{P}^T = (\ker d\check{\mu})_\mathbb{C}$ and $\mathcal{P}^{\check{T}} = (\ker d\mu)_\mathbb{C}$, where $\mu: X \to T$ and $\check{\mu}: X \to \check{T}$ are canonical projections. We construct a Toeplitz-type operator
\begin{align*}
	Q^\hbar: \mathcal{C}^\infty(X, \mathbb{C}) \to \operatorname{End}_\mathbb{C} \mathcal{H}^\hbar
\end{align*}
acting on the quantum Hilbert space $\mathcal{H}^\hbar$ in real polarization $\mathcal{P}^T$, together with a deformation quantization $\star$ on $(X, \omega)$ \emph{with separation of variables} in $(\mathcal{P}^{\check{T}}, \mathcal{P}^T)$ in the sense of Definition \ref{Definition 3.1}. For $f, g \in \mathcal{C}^\infty(X, \mathbb{C})$, let $f \star_N g$ be the $N$th order truncation of the power series $f \star g$ in $\hbar$. Our first main result shows that $(\mathcal{C}^\infty(X, \mathbb{C})[[\hbar]], \star)$ acts on $\mathcal{H}^\hbar$ asymptotically as $\hbar = \tfrac{1}{k} \to 0^+$.

\begin{theorem}
	\label{Theorem 1.1}
	For all $f, g \in \mathcal{C}^\infty(X, \mathbb{C})$ and $N \in \mathbb{N} \cup \{0\}$, there exists $K_N(f, g) > 0$ such that
	\begin{equation}
	\label{Equation 1.1}
	\left\lVert Q_f^\hbar \circ Q_g^\hbar - Q_{f \star_N g}^\hbar \right\rVert \leq K_N(f, g) \hbar^{N+1}
	\end{equation}
	for all $k \in \mathbb{N}$, where $\hbar = \tfrac{1}{k}$.
\end{theorem}

In addition, it is shown in Corollary \ref{Corollary 6.4} that our construction of the Toeplitz-type operator $Q^\hbar$ gives rise to representations of the quantum tori induced by the symplectic structure $\omega$ on $X$. This resulting representation is also seen in \cite{AZ2005}.\par
There are two main differences to the compact K\"ahler case. First, we need an appropriate extension of the $L^2$ inner product of smooth sections of the prequantum line bundle $L^{\otimes k}$ for defining the operators $Q^\hbar$, in which the quantum Hilbert space $\check{\mathcal{H}}^\hbar$ in the real polarization $\mathcal{P}^{\check{T}}$ transversal to $\mathcal{P}^T$ plays a crucial role. Second, our proof makes use of the estimation of the decay rates of Fourier coefficients of smooth functions on $X$.\par
If we reverse the roles of $\mathcal{P}^T$ and $\mathcal{P}^{\check{T}}$, we obtain another star product $\check{\star}$ and Toeplitz-type operators $\check{Q}^\hbar: \mathcal{C}^\infty(X, \mathbb{C}) \to \operatorname{End}_\mathbb{C} \check{\mathcal{H}}^\hbar$, `opposite to' $\star$ and $Q^\hbar$ respectively. There are Fourier-type transformations relating quantizations using $\mathcal{P}^T$ and $\mathcal{P}^{\check{T}}$ for both deformation quantization $\mathcal{F}^{\operatorname{DQ}}: (\mathcal{C}^\infty(X, \mathbb{C})[[\hbar]], \check{\star}) \overset{\cong}{\to} (\mathcal{C}^\infty(X, \mathbb{C})[[\hbar]], \star)$ and geometric quantization $B^\hbar: \check{\mathcal{H}}^\hbar \to \mathcal{H}^\hbar$ (see Proposition \ref{Proposition 3.3} and (\ref{Equation 6.8})). It is natural to ask whether these Fourier-type transforms are compatible in the sense that the following diagram commutes:
\begin{equation}
\label{Equation 1.2}
\begin{tikzcd}
(\mathcal{C}^\infty(X, \mathbb{C})[[\hbar]], \check{\star}) \ar[r, "\check{Q}^\hbar"] \ar[d, "\mathcal{F}^{\operatorname{DQ}}"'] & (\operatorname{End}_\mathbb{C} \check{\mathcal{H}}^\hbar, \circ) \ar[d, "\mathcal{F}^{\operatorname{GQ}}"]\\
(\mathcal{C}^\infty(X, \mathbb{C})[[\hbar]], \star) \ar[r, "Q^\hbar"'] & (\operatorname{End}_\mathbb{C} \mathcal{H}^\hbar, \circ)
\end{tikzcd}
\end{equation}
where $\mathcal{F}^{\operatorname{GQ}}$ is the conjugation by $B^\hbar$. Note that $\hbar = \tfrac{1}{k}$ with $k \in \mathbb{N}$ for geometric quantization whilst $\hbar$ is only a formal variable for deformation quantization. Let $\mathcal{F}_N^{\operatorname{DQ}}$ be the $N$th order truncation of the formal operator $\mathcal{F}^{\operatorname{DQ}}$ in $\hbar$. Our second main result asserts that (\ref{Equation 1.2}) commutes asymptotically as $\hbar \to 0^+$.


\begin{theorem}
	\label{Theorem 1.2}
	For all $f \in \mathcal{C}^\infty(X, \mathbb{C})$ and $N \in \mathbb{N} \cup \{0\}$, there exists $K > 0$ such that
	\begin{equation}
	\label{Equation 1.3}
	\lVert (\mathcal{F}^{\operatorname{GQ}} \circ \check{Q}^\hbar)(f) - (Q^\hbar \circ \mathcal{F}_N^{\operatorname{DQ}})(f) \rVert \leq K \hbar^{N+1}.
	\end{equation}
	for all $k \in \mathbb{N}$, where $\hbar = \tfrac{1}{k}$.
\end{theorem}

Deformation quantization on a $2n$-dimensional compact symplectic manifold always has a trace map
\begin{align*}
	\operatorname{Tr}: \mathcal{C}^\infty(X, \mathbb{C})[[\hbar]] \to \hbar^{-n} \mathbb{C}[[\hbar]]
\end{align*}
having the defining properties that $\operatorname{Tr}(f \star g) = \operatorname{Tr}(g \star f)$ and $\operatorname{Tr}(f) = \int_X f e^{\omega/\hbar} + \operatorname{O}(\hbar^{1-n})$ \cite{NT1995}. Physically, it is the partition function of quantum observables. This trace map is recovered by Grady-Li-Li \cite{GLL2017} via \emph{Batalin-Vilkovisky quantization} \cite{BV1984} in Costello-Gwilliam approach \cite{C2011, CG2017}. In K\"ahler case, it is shown in \cite{BMS1994} that for all $f \in \mathcal{C}^\infty(X, \mathbb{C})$, 
\begin{equation}
\operatorname{tr} Q_f^\hbar = \hbar^{-n} \left( \int_X f e^\omega + \operatorname{O}(\hbar) \right),
\end{equation}
and its asymptotic expansion gives a trace map for the Berezin-Toeplitz star product $\star^{\operatorname{BT}}$ \cite{S2000}. Our third main result gives the real counterpart of the above statements.

\begin{theorem}
	\label{Theorem 1.3}
	Let $X = T \times \check{T}$, where $T$ is an affine torus with dual torus $\check{T}$ and $\omega$ be the standard symplectic form on $X$. If $f \in \mathcal{C}^\infty(X, \mathbb{C})$, then
	\begin{equation}
	\operatorname{tr}( Q_f^\hbar ) = \hbar^{-n} \left( \int_X fe^\omega + \operatorname{O}(\hbar^\infty) \right),
	\end{equation}
	and its asymptotic expansion is a trace map for the star product $\star$.
\end{theorem}

The proofs of our main results have two stages. In the first stage, we argue  with the aid of microlocal analysis that by transversality of the two real polarizations $\mathcal{P}^T, \mathcal{P}^{\check{T}}$, we can extend the $L^2$ inner product to a pairing between two suitable classes of distributional sections of $L^{\otimes k}$. We also construct a distributional kernel, which is the real analogue to Bergman kernel for K\"ahler polarizations, so as to obtain a projection $\Pi^\hbar$ onto the quantum Hilbert space $\mathcal{H}^\hbar$. The Toeplitz-type operator is given by $Q_f^\hbar s = \Pi^\hbar (fs)$. In the second stage, noting that all our main results are essentially norm estimations, we define and estimate various norms of operators on $\mathcal{H}^\hbar$, serving as an intermediate steps towards the estimation of the desired norms. These estimations are due to the fact that fibrewise Fourier coefficients of smooth functions on the torus bundle $X \to \check{T}$ has rapid decay uniform on $\check{T}$. For Theorems \ref{Theorem 1.1} and \ref{Theorem 1.2}, in particular, we see that the operators in $\lVert \quad \rVert$ in (\ref{Equation 1.1}) and (\ref{Equation 1.2}) involve remainder terms of Taylor expansion of smooth functions on $X$ which can be expressed in terms of fibrewise Fourier coefficients.\par
The paper is organized as follows. In Section \ref{Section 2}, we give a quick review of Berezin-Toeplitz quantization on compact K\"ahler manifolds. Sections \ref{Section 3}, \ref{Section 4} and \ref{Section 5} are preparations for proofs of our theorems, including deformation quantization with separation of variables, geometric quantization of $(X, \omega)$ by the work \cite{BMN2010}, and fibrewise Fourier transform and Weil-Brezin transform \cite{F1989}. Eventually in Section \ref{Section 6}, we provide proofs of our theorems.

\subsection{Acknowledgement}
\quad\par
We thank Kwok Wai Chan, Qin Li, Si Li and Ziming Nikolas Ma for useful comments as well as helpful discussions. This research was substantially supported by grants from the Research Grants Council of the Hong Kong Special Administrative Region, China (Project No. CUHK14301619 and CUHK14306720) and direct grants from the Chinese University of Hong Kong.

\section{Review of Berezin-Toeplitz Quantization on Compact K\"ahler Manifolds}
\label{Section 2}
This section gives a short review of the Berezin-Toeplitz quantization on compact K\"ahler manifolds. Readers familiar with this can skip this section.\par
\emph{Berezin-Toeplitz (BT) quantization} is a quantization scheme on a compact K\"ahler manifold $(X, \omega, J)$, where $[\omega] = c_1(L) \in H^2(X, \mathbb{Z})$ with $L$ a prequantum line bundle, realizing an asymptotic action of deformation quantization on geometric quantization. The complex structure $J$ induces a complex polarization $\mathcal{P} = T^{1, 0}X$ and it yields the quantum Hilbert space $\mathcal{H}^\hbar = H^0(X, L^{\otimes k})$ for each positive integer $k \in \mathbb{N}$ with $\hbar = \tfrac{1}{k}$. In this case, the \emph{Toeplitz operator} is given by
\begin{equation}
	\label{Equation 2.1}
	Q^\hbar: \mathcal{C}^\infty(X, \mathbb{C}) \times \mathcal{H}^\hbar \to \mathcal{H}^\hbar, \quad (f, s) \mapsto Q_f^\hbar s := \Pi^\hbar (fs),
\end{equation}
where $\Pi^\hbar: \Gamma_{L^2}(X, L^{\otimes k}) \to \mathcal{H}^\hbar$ is the orthogonal projection with respect to the $L^2$-inner product. There exists a unique deformation quantization $\star^{\operatorname{BT}}$, known as the \emph{Berezin-Toepltiz star product}, such that we can find a constant $K_N(f, g)$ independent of $k$ making the following estimates hold:
\begin{equation}
	\left\lVert Q_f^\hbar \circ Q_g^\hbar - Q_{f \star_N^{\operatorname{BT}} g}^\hbar \right\rVert \leq K_N(f, g) \hbar^{N+1},
\end{equation}
where $\hbar = \tfrac{1}{k}$ and $f \star_N^{\operatorname{BT}} g$ is the $N$th order truncation of the power series $f \star^{\operatorname{BT}} g$ in $\hbar$ \cite{BMS1994, KS2000, S2000}. After the discovery of the Toeplitz operator $Q^\hbar$ and the BT star product $\star^{\operatorname{BT}}$, it prompts numerous research in this topic, for instance \cite{CLQ2020a, CLQ2020b, CLQ2020c, M2004, NT2004, T2009}.\par 
In \cite{K1996}, 
Karabegov defined the notion of a deformation quantization \emph{with separation of variables} as a deformation quantization $\star$ on $(X, \omega, J)$ (more generally, on a pseudo-K\"ahler manifold) such that for all open subset $U$ of $X$ and $f, g \in \mathcal{C}^\infty(U, \mathbb{C})$, if $f$ is holomorphic or $g$ is anti-holomorphic, then
\begin{equation}
	f \star g = fg.
\end{equation}
There is a one-to-one correspondence between deformation quantizations with separation of variables on $(M, \omega, J)$ and closed formal $(1, 1)$-forms $\hbar^{-1} \omega + \sum_{i=0}^\infty \hbar^i \omega_i \in \hbar^{-1}\Omega_{\operatorname{cl}}^{1, 1}(X)[[\hbar]]$, and there introduced a \emph{formal Berezin transform} in \cite{K1998} which transforms deformation quantization with separation of variables on $(X, \omega, J)$ to those on the pseudo-K\"ahler manifold $(X, \omega, -J)$, i.e. those with the roles of holomorphic and antiholomorphic variables swapped. The Berezin-Toeplitz star product $\star^{\operatorname{BT}}$ is a deformation quantization with separation of variables on $(X, \omega, -J)$, therefore through the formal Berezin transform we obtain a star product $\check{\star}^{\operatorname{BT}}$ out of $\star^{\operatorname{BT}}$ and an isomorphism of star products
\begin{equation}
	\mathcal{F}^{\operatorname{DQ}}: (\mathcal{C}^\infty(X, \mathbb{C})[[\hbar]], \check{\star}^{\operatorname{BT}}) \to (\mathcal{C}^\infty(X, \mathbb{C})[[\hbar]], \star^{\operatorname{BT}}).
\end{equation}
Karabegov and Schilchenmaier \cite{KS2000} found that $\mathcal{F}^{\operatorname{DQ}}$ can be reconstructed from the Toeplitz operator $Q^\hbar$. They defined the \emph{Berezin transform} $\mathcal{F}^\hbar: \mathcal{C}^\infty(X, \mathbb{C}) \to \mathcal{C}^\infty(X, \mathbb{C})$ by $f \mapsto \sigma^\hbar(Q_f^\hbar)$, where $\sigma^\hbar(A)$ is the \emph{Berezin's covariant symbol} of an operator $A \in \operatorname{End}_\mathbb{C}\mathcal{H}^\hbar$, and Theorem 5.9 in \cite{KS2000} shows that the asymptotic expansion of $\mathcal{F}^\hbar$ gives the formal Berezin transform $\mathcal{F}^{\operatorname{DQ}}$.\par
As a remark, we can extend Karabegov's notion of deformation quantization with separation of variables on an arbitrary symplectic manifolds $(X, \omega)$. For any polarization $\mathcal{P}$ on $(X, \omega)$ and an open subset $U$ of $X$, we define the space of $\mathcal{P}$-invariant functions on $U$:
\begin{align*}
\mathcal{O}_{\mathcal{P}}(U) = \{ f \in \mathcal{C}^\infty(U, \mathbb{C}): \mathcal{L}_\xi f = 0 \text{ for any } \xi \in \Gamma^\infty(U, \mathcal{P}) \}.
\end{align*}

\begin{definition}
	\label{Definition 3.1}
	Let $\mathcal{P}, \mathcal{Q}$ be polarizations of $(X, \omega)$ transversal to each other, i.e. $TX_\mathbb{C} = \mathcal{P} \oplus \mathcal{Q}$. A deformation quantization $\star$ of $(X, \omega)$ is said to be \emph{with separation of variables in} $(\mathcal{P}, \mathcal{Q})$ if for all open subset $U$ of $X$ and $f, g \in \mathcal{C}^\infty(U, \mathbb{C})$, if either $f \in \mathcal{O}_\mathcal{P}(U)$ or $g \in \mathcal{O}_\mathcal{Q}(U)$, then
	\begin{align*}
	f \star g = fg.
	\end{align*}
\end{definition}

\section{Deformation Quantization Compatible with Polarizations on Symplectic Tori}
\label{Section 3}

In this section, we shall define star products on the symplectic torus $X = \mathbb{R}^{2n} / \mathbb{Z}^{2n}$ with periodic coordinates $(x, y)$ with $x, y \in \mathbb{R}^n$ and symplectic form $\omega = \sum_{i=1}^n dy^i \wedge dx^i$ that appear in our main theorems. Define the fibrations $\mu: X \to T$ by $(x, y) + \mathbb{Z}^{2n} \mapsto x + \mathbb{Z}^n$ and $\check{\mu}: X \to \check{T}$ by $(x, y) + \mathbb{Z}^{2n} \mapsto y + \mathbb{Z}^n$, where $T$ and $\check{T}$ are two copies of $\mathbb{R}^n/\mathbb{Z}^n$, yielding transversal real polarizations $\mathcal{P}^T = (\ker d\check{\mu})_\mathbb{C}$ and $\mathcal{P}^{\check{T}} = (\ker d\mu)_\mathbb{C}$.\par
The tangent bundle $TX$ is canonically trivialized. Now we denote by $\Gamma^\infty(X, TX_\mathbb{C}^{\otimes 2})_{\operatorname{const}, \omega}$ the space of $2$-tensor fields $\alpha \in \Gamma^\infty(X, TX_\mathbb{C} \otimes TX_\mathbb{C})$ of constant coefficients, which are regarded as operators $\mathcal{C}^\infty(X, \mathbb{C}) \otimes \mathcal{C}^\infty(X, \mathbb{C}) \to \mathcal{C}^\infty(X, \mathbb{C}) \otimes \mathcal{C}^\infty(X, \mathbb{C})$, satisfying
\begin{equation}
\operatorname{Mult} \circ (\alpha(f, g) - \alpha(g, f)) = \frac{\sqrt{-1}}{2\pi} \{f, g\},
\end{equation}
where $\operatorname{Mult}$ is the usual commutative product on $\mathcal{C}^\infty(X, \mathbb{C})$. Thus, for such a $2$-tensor field $\alpha$, its antisymmetric part is fixed by the Poisson bivector $\omega^{-1}$ which is the inverse of $\omega$ and is given by $\sum_{i=1}^n \partial_{x^i} \wedge \partial_{y^i}$ - we can only vary its symmetric part. We now have a special kind of deformation quantizations on $(X, \omega)$, namely those of the form
\begin{equation}
\label{Equation 3.2}
f \star_\alpha g = \operatorname{Mult} \circ e^{\hbar \alpha} (f \otimes g),
\end{equation}
for some $\alpha \in \Gamma^\infty(X, TX_\mathbb{C}^{\otimes 2})_{\operatorname{const}, \omega}$. In formula (\ref{Equation 3.2}), $\operatorname{Mult}$ is extended by $\mathbb{C}[[\hbar]]$-linearity. This kind of construction of star products appeared, for example, as fibrewise Wick products on the Weyl bundle of an almost K\"ahler manifold in a variant of Fedosov’s quantization \cite{F1996} described by Karabegov and Schlichenmaier in \cite{KS2001}.\par 
Indeed, for all $\alpha, \beta \in \Gamma^\infty(X, TX_\mathbb{C}^{\otimes 2})_{\operatorname{const}, \omega}$, the star products $\star_\alpha, \star_\beta$ are in the same equivalence class. Moreover, we can describe certain isomorphisms between them explicitly. This is related to the difference $\beta - \alpha$, which is symmetric. For any symmetric $2$-tensor field $\gamma \in \Gamma^\infty(X, TX_\mathbb{C} \otimes TX_\mathbb{C})$ of constant coefficients, we denote by $\partial_\gamma$ the second order differential operator on $X$ given by
\begin{align*}
\partial_\gamma = \sum_{i=1}^{2n} \sum_{j=1}^{2n} \gamma^{ij} \frac{\partial^2}{\partial x^i \partial x^j},
\end{align*}
where we write $\gamma = \sum_{i=1}^{2n} \sum_{j=1}^{2n} \gamma^{ij} \partial_{x^i} \otimes \partial_{x^j}$ and denote $y^i$ by $x^{n+i}$ for $i \in \{1, ..., n\}$. The following proposition should be known among experts (for instance one can refer to Lemma 7 in \cite{C2016}), but we shall state it so as to make our discussion self-contained. 

\begin{proposition}
	\label{Proposition 3.2}
	Let $\alpha, \beta \in \Gamma^\infty(X, TX_\mathbb{C}^{\otimes 2})_{\operatorname{const}, \omega}$. Then
	\begin{align*}
	e^{\frac{1}{2} \hbar \partial_{\beta - \alpha}}: (\mathcal{C}^\infty(X, \mathbb{C})[[\hbar]], \star_\alpha) \to (\mathcal{C}^\infty(X, \mathbb{C})[[\hbar]], \star_\beta)
	\end{align*}
	is an isomorphism of star products.
\end{proposition}

There is a canonical choice $\tilde{\omega}^{-1} = \tfrac{\sqrt{-1}}{4\pi} \sum_{i=1}^n \left( \partial_{x^i} \otimes \partial_{y^i} - \partial_{y^i} \otimes \partial_{x^i} \right)$ in $\Gamma^\infty(X, TX_\mathbb{C}^{\otimes 2})_{\operatorname{const}, \omega}$. The corresponding deformation quantization $\star_{\tilde{\omega}^{-1}}$ is known as the \emph{Moyal product}. Unfortunately, for any pair of transversal translation-invariant polarizations $(\mathcal{P}, \mathcal{Q})$, the Moyal product is undesirable in the sense that it is not with separation of variables in $(\mathcal{P}, \mathcal{Q})$. Instead, the pair $(\mathcal{P}, \mathcal{Q})$ determines a $2$-tensor field $\alpha \in \Gamma^\infty(X, TX_\mathbb{C}^{\otimes 2})_{\operatorname{const}, \omega}$ lying in $\Gamma^\infty(X, \mathcal{P} \otimes \mathcal{Q})$. The decomposition $TX_\mathbb{C} = \mathcal{P} \oplus \mathcal{Q}$ induces a projection $TX_\mathbb{C} \otimes TX_\mathbb{C} \to \mathcal{P} \otimes \mathcal{Q}$, and $\alpha$ is indeed obtained by the projection of $\tilde{\omega}^{-1}$. Clearly, $\star_\alpha$ is with separation of variables in $(\mathcal{P}, \mathcal{Q})$. If we swap the order and consider the pair $(\mathcal{Q}, \mathcal{P})$, we have the respective $2$-tensor field $\tilde{\alpha} \in \Gamma^\infty(X, TX_\mathbb{C}^{\otimes 2})_{\operatorname{const}, \omega}$. In this case, we call $\star_\alpha, \star_{\tilde{\alpha}}$ \emph{opposite star products} of each other.\par
By the above explanation, as we shall consider polarizations $\mathcal{P}^T, \mathcal{P}^{\check{T}}$, it is natural to define a star product $\star$ on $\mathcal{C}^\infty(X, \mathbb{C})[[\hbar]]$ as follows: for all $f, g \in \mathcal{C}^\infty(X, \mathbb{C})$,
\begin{equation}
\label{Equation 3.3}
f \star g = \sum_{k=0}^\infty \hbar^k C_k(f, g),
\end{equation}
where for any $k \in \mathbb{N} \cup \{0\}$,
\begin{equation}
\label{Equation 3.4}
C_k(f, g) = \frac{1}{k! \cdot (2\pi \sqrt{-1})^k} \sum_{i_1=1}^n \cdots \sum_{i_k=1}^n \frac{\partial^k f}{\partial y^{i_1} \cdots \partial y^{i_k}} \frac{\partial^k g}{\partial x^{i_1} \cdots \partial x^{i_k}} = \frac{1}{(2\pi \sqrt{-1})^k} \sum_{\lvert I \rvert = k} \frac{1}{I!} \frac{\partial^k f}{\partial y^I} \frac{\partial^k g}{\partial x^I}.
\end{equation}
And we define another star product $\check{\star}$ as follows: for all $f, g \in \mathcal{C}^\infty(X, \mathbb{C})$,
\begin{equation}
\label{Equation 3.5}
f \check{\star} g = \sum_{k=0}^\infty \hbar^k \check{C}_k(f, g),
\end{equation}
where for any $k \in \mathbb{N} \cup \{0\}$,
\begin{equation}
\check{C}_k(f, g) = \frac{(\sqrt{-1})^k}{(2\pi)^k k!} \sum_{i_1=1}^n \cdots \sum_{i_k=1}^n \frac{\partial^k f}{\partial x^{i_1} \cdots \partial x^{i_k}} \frac{\partial^k g}{\partial y^{i_1} \cdots \partial y^{i_k}} = \frac{(\sqrt{-1})^k}{(2\pi)^k} \sum_{\lvert I \rvert = k} \frac{1}{I!} \frac{\partial^k f}{\partial x^I} \frac{\partial^k g}{\partial y^I}.
\end{equation}
The following proposition is a direct application of Proposition \ref{Proposition 3.2}.

\begin{proposition}
	\label{Proposition 3.3}
	The star product $\star$ defined as in (\ref{Equation 3.3}) (resp. $\check{\star}$ defined as in (\ref{Equation 3.5})) is a deformation quantization on $(X, \omega)$ with separation of variables in the pair of transversal polarizations $(\mathcal{P}^{\check{T}}, \mathcal{P}^T)$ (resp. $(\mathcal{P}^T, \mathcal{P}^{\check{T}})$). Also, $\star$ and $\check{\star}$ are opposite star products and
	\begin{align*}
	e^{\hbar \Delta}: (\mathcal{C}^\infty(X)[[\hbar]], \star) \to (\mathcal{C}^\infty(X)[[\hbar]], \check{\star})
	\end{align*}
	is an isomorphism of star products, where $\Delta = \tfrac{\sqrt{-1}}{2\pi} \sum_{i=1}^n \tfrac{\partial^2}{\partial x^i \partial y^i}$. 
\end{proposition}

We end this section by discussing trace maps for deformation quantizations.

\begin{definition}
	\label{Definition 3.4}
	A \emph{trace map} for a deformation quantization $\star$ of $(X, \omega)$ is a $\mathbb{C}[[\hbar]]$-linear continuous functional $\operatorname{Tr}: \mathcal{C}^\infty(X, \mathbb{C})[[\hbar]] \to \hbar^{-n} \mathbb{C}[[\hbar]]$ such that
	\begin{enumerate}
		\item $\operatorname{Tr}(f \star g) = \operatorname{Tr}(g \star f)$ for all $f, g \in \mathcal{C}^\infty(X, \mathbb{C})$; and
		\item for all $f \in \mathcal{C}^\infty(X, \mathbb{C})$, $\operatorname{Tr}(f) = \int_X f e^{\omega/\hbar} + \operatorname{O}(\hbar^{1-n})$.
	\end{enumerate}
\end{definition}

Define
\begin{align}
\label{Equation 3.7}
\operatorname{Tr}: \mathcal{C}^\infty(X, \mathbb{C})[[\hbar]] \to \hbar^{-n} \mathbb{C}[[\hbar]], \quad f \mapsto \operatorname{Tr}(f) := \int_X f e^{\omega/\hbar}.
\end{align}


\begin{proposition}
	The map $\operatorname{Tr}$ defined as in (\ref{Equation 3.7}) is a trace map for both the star products $\star, \check{\star}$ defined as in (\ref{Equation 3.3}) and (\ref{Equation 3.5}) respectively.
\end{proposition}
\begin{proof}
	Let $f, g \in \mathcal{C}^\infty(X, \mathbb{C})[[\hbar]]$. Observe that for all $i \in \{1, ..., n\}$, $\mathcal{L}_{\partial_{x^i}} \omega = \mathcal{L}_{\partial_{y^i}} \omega = 0$. Thus, for all multi-index $I$, by Stoke's Theorem,
	\begin{align*}
	\int_X \frac{\partial^{\lvert I \rvert} f}{\partial x^I} \frac{\partial^{\lvert I \rvert} g}{\partial y^I} e^{\omega/\hbar} = \int_X \frac{\partial^{\lvert I \rvert} g}{\partial x^I} \frac{\partial^{\lvert I \rvert} f}{\partial y^I} e^{\omega/ \hbar}.
	\end{align*}
	It directly follows from the above observation that $\operatorname{Tr}(f \star g) = \operatorname{Tr}(g \star f)$ and $\operatorname{Tr}(f \check{\star} g) = \operatorname{Tr}(g \check{\star} f)$. Also, $\operatorname{Tr}$ clearly saticfies the second condition in Definition \ref{Definition 3.4}.
\end{proof}

\section{Geometric Quantization on Symplectic Tori}
\label{Section 4}
In this section, we shall perform geometric quantization on the symplectic torus $X = \mathbb{R}^{2n}/\mathbb{Z}^{2n}$ in translation-invariant non-negative polarizations and introduce the BKS pairing between different quantum Hilbert spaces, which is shown to be important for the construction of Toeplitz-type operators for real polarizations in Section \ref{Section 6}.

\subsection{Prequantum line bundle}
\quad\par
We consider a prequantum line bundle $L$ equipped with a Hermitian connection $\nabla$ with curvature $-2\pi\sqrt{-1}\omega$. Explicitly, the line bundle $L$ is constructed as follows. The trivial line bundle $\mathbb{R}^{2n} \times \mathbb{C}$ on $\mathbb{R}^{2n}$ has a canonical Hermitian metric given by $\langle (u, \xi), (u, \xi') \rangle = (u, \xi \overline{\xi'})$ for all $u \in \mathbb{R}^{2n}$ and $\xi, \xi' \in \mathbb{C}$. This trivial line bundle has a Hermitian connection
\begin{equation}
\nabla = d - \pi \sqrt{-1} \sum_{i=1}^n (y^i dx^i - x^i dy^i).
\end{equation}
The group $\mathbb{Z}^{2n}$ acts on $\mathbb{R}^{2n} \times \mathbb{C}$ by
\begin{equation}
\lambda \cdot (u, \xi) = (u + \lambda, \alpha(\lambda) e^{-\pi \sqrt{-1} \omega(u, \lambda)} \xi),
\end{equation}
for all $u \in \mathbb{R}^{2n}$ and $\lambda \in \mathbb{Z}^{2n}$, where $\alpha: \mathbb{Z}^{2n} \to \{-1, 1\}$ is the so-called ``canonical'' semi-character
\begin{align*}
\alpha(\lambda) = (-1)^{\sum_{i=1}^n \lambda_i \lambda_{n+i}}
\end{align*}
and $\omega$ is identified with symplectic bilinear form
\begin{align*}
\omega(u, v) = \sum_{i=1}^n (u_{n+i} v_i - u_iv_{n+i}).
\end{align*}
We define $L$ to be the quotient $\mathbb{R}^{2n} \times_{\mathbb{Z}^{2n}} \mathbb{C}$, and both the Hermitian metric and the Hermitian connection on $\mathbb{R}^{2n} \times \mathbb{C}$ descend to $L$.\par
Because we pass from deformation quantization to geometric quantization, instead of regarding $\hbar$ as a formal variable, we set $\hbar = \tfrac{1}{k}$ for $k \in \mathbb{N}$ \cite{D2001}. Then $L^{\otimes k}$ is equipped with the Hermitian connection $\nabla^\hbar = d - k \pi\sqrt{-1} \sum_{i=1}^n (y^i dx^i - x^i dy^i)$ with curvature $-2k\pi\sqrt{-1} \omega$. Here and below, the space $\Gamma^\infty(X, L^{\otimes k})$ of global smooth sections of $L^{\otimes k}$ will be identified with the space $\mathcal{C}_{\operatorname{qper}, k}^\infty(\mathbb{R}^{2n}, \mathbb{C})$ of smooth functions $s \in \mathcal{C}^\infty(\mathbb{R}^{2n}, \mathbb{C})$ that satisfy the appropriate quasi-periodicity conditions:
\begin{equation}
\label{Equation 4.3}
s(u + \lambda) = [\alpha(\lambda)]^k e^{-k\pi \sqrt{-1} \omega(u, \lambda)} s(u),
\end{equation}
for all $u \in \mathbb{R}^{2n}$ and $\lambda \in \mathbb{Z}^{2n}$. We have an inner product $\langle \quad, \quad \rangle_{L^2}$ on $L^{\otimes k}$ induced by the Hermitian metric on $L^{\otimes k}$ and Liouville measure on $X$.

\subsection{The polarized sections}
\quad\par
Translation-invariant non-negative polarizations on $(X, \omega)$ are parametrized by the closed Siegel disc $\mathbb{D}_n$, i.e. the space of symmetric $n \times n$ matrices $A$ over $\mathbb{C}$ such that $I_n - A\overline{A}$ is positive semidefinite. For $\tau \in \mathbb{D}_n$,
we have the respective translation-invariant polarization given by: for all $p \in X$,
\begin{align*}
\mathcal{P}_p^\tau = \operatorname{span}_\mathbb{C} \left\{ \sum_{i=1}^n \left( -\sqrt{-1} (I_n - \overline{\tau})_{ij} \left. \frac{\partial}{\partial x^i} \right\vert_p + (I_n + \overline{\tau})_{ij} \left. \frac{\partial}{\partial y^i} \right\vert_p \right): j \in \{1, ..., n\} \right\},
\end{align*}
where $I_n$ is the $n \times n$ identity matrix. Let $\mathring{\mathbb{D}}_n$ be the interior of $\mathbb{D}_n$ and $\mathbb{H}_n$ be the Siegel upper-half space, i.e. the space of symmetric $n \times n$ matrices $A$ over $\mathbb{C}$ whose imaginary parts $\operatorname{Im} A$ are positive definite. We have the Cayley transform
\begin{equation}
\mathbb{H}_n \to \mathring{\mathbb{D}}_n, \quad \Omega \mapsto \tau(\Omega) = (\sqrt{-1}I_n - \Omega) (\sqrt{-1}I_n + \Omega)^{-1}
\end{equation}
with inverse
\begin{equation}
\mathring{\mathbb{D}}_n \to \mathbb{H}_n, \quad \tau \mapsto \Omega(\tau) = \sqrt{-1} (I_n - \tau) (I_n + \tau)^{-1}.
\end{equation}
For $\tau \in \mathring{\mathbb{D}}_n$, $\mathcal{P}^\tau$ is a positive polarization with respect to a complex coordinate $z_\tau := x - \Omega(\tau) y$, giving a complex manifold $X_\tau$. The two polarizations $\mathcal{P}^T, \mathcal{P}^{\check{T}}$ of our concern correspond to two points at the boundary of $\mathbb{D}_n$, namely $\mathcal{P}^T = \mathcal{P}^{-I_n}$ and $\mathcal{P}^{\check{T}} = \mathcal{P}^{I_n}$.\par 
As quantum Hilbert spaces in real polarizations involve distributional sections of $L^{\otimes k}$ in general, we need to deal with the Gelfand triple
\begin{align*}
	\Gamma^\infty(X, L^{\otimes k}) \hookrightarrow \Gamma_{L^2}(X, L^{\otimes k}) \cong \Gamma_{L^2}(X, L^{\otimes k})' \hookrightarrow \Gamma^{-\infty}(X, L^{\otimes k}) \cong (\Gamma^\infty(X, \overline{L}^{\otimes k}))'.
\end{align*}
For a distributional section $s \in \Gamma^{-\infty}(X, L^{\otimes k})$ and a smooth section $\tau \in \Gamma^\infty(X, L^{\otimes k})$, we define $\langle s, \tau \rangle = s(\overline{\tau})$ and $\langle \tau, s \rangle = \overline{\langle s, \tau \rangle}$. We consider the embedding induced by the Liouville measure
\begin{align*}
	\iota: \Gamma^\infty(X, L^{\otimes k}) \hookrightarrow \Gamma^{-\infty}(X, L^{\otimes k})
\end{align*}
given by $\langle \iota(s), \tau \rangle = \langle s, \tau \rangle_{L^2}$ for any $\tau \in \Gamma^\infty(X, L^{\otimes k})$. Preserving the embedding $\iota$, we can extends actions of differential operators on smooth sections of $L^{\otimes k}$ to distributional sections via `integration by parts'. In particular, for arbitrary $s \in \Gamma^{-\infty}(X, L^{\otimes k})$ and test section $\tau \in \Gamma^\infty(X, L^{\otimes k})$, we have (i) $\langle fs, \tau \rangle := \left\langle s, \overline{f} \tau \right\rangle$ for all $f \in \mathcal{C}^\infty(X, \mathbb{C})$; (ii) $\left\langle \nabla_\xi^\hbar s, \tau \right\rangle = -\left\langle (\operatorname{div}_\omega \xi) s, \tau \right\rangle - \left\langle s, \nabla_\xi^\hbar \tau \right\rangle$ for all $\xi \in \mathcal{X}(X)_\mathbb{C}$, where $\operatorname{div}_\omega \xi = \frac{\mathcal{L}_\xi \omega^n}{\omega^n}$ is the divergence of $\xi$ with respect to $\tfrac{1}{n!} k^n \omega^n$.\par
For a polarization $\mathcal{P}$ on $(X, \omega)$, by a \emph{$\mathcal{P}$-polarized (distributional) section} of $L^{\otimes k}$, we mean a distributional section $s \in \Gamma^{-\infty}(X, L^{\otimes k})$ such that $\nabla_{\overline{\xi}}^\hbar s = 0$ for any $\xi \in \Gamma^\infty(X, \mathcal{P})$. 

\subsection{The half-form correction}
\quad\par
We first introduce an indefinite pairing on the space of $n$-forms on $X$ given by
\begin{equation}
\langle \eta, \eta' \rangle = \frac{n! \cdot \sqrt{-1}^n (-1)^{\frac{n(n-1)}{2}} \eta \wedge \overline{\eta'}}{(2k\omega)^n}.
\end{equation}
For $\tau \in \mathbb{D}_n$, define the complex line bundle $K^\tau$ on $X$ as follows: for all $p \in X$,
\begin{equation}
	K_p^\tau = \left\{ \alpha \in \textstyle\bigwedge^n T_p^*X_\mathbb{C}: \iota_{\overline{\xi}} \alpha = 0 \text{ for any } \xi \in \mathcal{P}_p^\tau \right\}.
\end{equation}
This canonical line bundle $K^\tau$ is generated by $d^n(x, y)_\tau := \bigwedge^n ((I_n + \tau) dx - \sqrt{-1} (I_n - \tau) dy)$. When $\tau \in \mathring{\mathbb{D}}_n$, $K^\tau$ is just the canonical line bundle of the complex manifold $X_\tau$. We see that
\begin{align*}
	\left\langle d^n(x, y)_\tau, d^n(x, y)_{\tau'} \right\rangle = \det(I_n + \tau) \det(I_n + \overline{\tau'}) \det \frac{\Omega(\tau) - \overline{\Omega(\tau')}}{2k\sqrt{-1}}, \quad \tau, \tau' \in \mathring{\mathbb{D}_n}.
\end{align*}
By direct calculation,
\begin{equation}
	\left\langle d^n(x, y)_{-I_n}, d^n(x, y)_{I_n} \right\rangle = 2^n k^{-n}.
\end{equation}
The half-form correction consists of the choices of the following data: (i) a square root $\sqrt{K^\tau}$ of $K^\tau$, i.e. $\sqrt{K^\tau} \otimes \sqrt{K^\tau} \cong K^\tau$, for all $\tau \in \mathbb{D}_n$; (ii) a pairing $\langle \quad, \quad \rangle: \sqrt{K^\tau} \otimes \sqrt{K^{\tau'}} \to X \times \mathbb{C}$ for all $\tau, \tau' \in \mathbb{D}_n$, such that the following diagram commutes:
\begin{center}
	\begin{tikzcd}
		(\sqrt{K^\tau} \otimes \sqrt{K^{\tau'}})^{\otimes 2} \ar[rrr, "\langle \quad{,} \quad \rangle \otimes \langle \quad{,} \quad \rangle"] \ar[d, "\cong"'] &&& (X \times \mathbb{C})^{\otimes 2} \ar[d, "\cong"]\\
		K^\tau \otimes K^{\tau'} \ar[rrr, "\langle \quad{,} \quad \rangle"] &&& X \times \mathbb{C}
	\end{tikzcd}
\end{center}
The essential ingredient for these data is the choice of the square root of $\left\langle d^n(x, y)_\tau, d^n(x, y)_{\tau'} \right\rangle$ for $\tau, \tau' \in \mathbb{D}_n$. We need to first fix some choices of branches of square roots.
\begin{itemize}
	\item The branch of the square root $\sqrt{\det(I_n + \tau)}$, $\tau \in \mathbb{D}_n$, is chosen to be such that it is $1$ when $\tau = 0$, and choose the branch of the square root $\sqrt{\det(I_n + \overline{\tau})} := \overline{\sqrt{\det(I_n + \tau)}}$.
	\item The branch of the square root of the determinant of the form $\left( \det A \right)^{-1}$, where $A$ is an invertible symmetric $n \times n$ matrix over $\mathbb{C}$, is naturally chosen to be
	\begin{equation*}
	\left( \det A \right)^{-1/2} := (2\pi)^{n/2} \int_{\mathbb{R}^n} e^{- \frac{\xi \cdot A \xi}{2}} d^n\xi.
	\end{equation*}
\end{itemize}
With these choices, we take the branch of the square root of $\left\langle d^n(x, y)_\tau, d^n(x, y)_{\tau'} \right\rangle$ for $\tau, \tau' \in \mathbb{D}_n$ so that for all $\tau, \tau' \in \mathring{\mathbb{D}_n}$,
\begin{equation}
	\sqrt{\left\langle d^n(x, y)_\tau, d^n(x, y)_{\tau'} \right\rangle} = \sqrt{\det (I_n + \tau)} \sqrt{\det (I_n + \overline{\tau'})} \left( \det \frac{\Omega(\tau) - \overline{\Omega(\tau')}}{2k\sqrt{-1}} \right)^{1/2}.
\end{equation}
We make the choice such that there is a generator $\sqrt{d^n(x, y)_\tau}$ of $\sqrt{K^\tau}$ for each $\tau \in \mathbb{D}_n$, and
\begin{align*}
\left\langle \sqrt{d^n(x, y)_\tau}, \sqrt{d^n(x, y)_{\tau'}} \right\rangle = \sqrt{\left\langle d^n(x, y)_\tau, d^n(x, y)_{\tau'} \right\rangle}, \quad \text{for any } \tau, \tau' \in \mathbb{D}_n,
\end{align*}
implying that the vector bundle map $\sqrt{K^\tau} \otimes \sqrt{K^\tau} \to K^\tau, \sqrt{d^n(x, y)_\tau} \otimes \sqrt{d^n(x, y)_\tau} \mapsto d^n(x, y)_\tau$ is an isomorphism preserving the Hermitian metrics. In particular,
\begin{equation}
\label{Equation 4.10}
\left\langle \sqrt{d^n(x, y)_{-I_n}}, \sqrt{d^n(x, y)_{I_n}} \right\rangle = 2^{\frac{n}{2}} k^{-\frac{n}{2}}.
\end{equation}
By \cite{BMN2010}, $\sqrt{d^n(x, y)_\tau}$ is parallel with respect to a chosen compatible $\overline{\mathcal{P}^\tau}$-connection on $\sqrt{K^\tau}$.

\subsection{The quantum Hilbert space}
\quad\par

\begin{definition}
	For $\tau \in \mathbb{D}_n$, the Hilbert space of quantum states is
	\begin{align*}
	\mathcal{H}_\tau^\hbar := \{ s \in \Gamma^{-\infty}(X, L^{\otimes k}): \nabla_{\overline{\xi}}^\hbar s = 0 \text{ for any } \xi \in \Gamma^\infty(X, \mathcal{P}^\tau) \} \otimes \mathbb{C} \sqrt{d^n(x, y)_\tau}.
	\end{align*}
\end{definition}

Given the above definition, we owe the readers the inner product on $\mathcal{H}_\tau^\hbar$. Indeed, we can define a more general pairing between different Hilbert spaces of quantum states $\mathcal{H}_\tau^\hbar, \mathcal{H}_{\tau'}^\hbar$ ($\tau, \tau' \in \mathbb{D}_n$), known as the \emph{Blattner-Kostant-Sternberg (BKS) pairing}. On the interior $\mathring{\mathbb{D}}_n$, the BKS pairing $\langle \quad, \quad \rangle_{\operatorname{BKS}}$ is given as follows: for $\tau, \tau' \in \mathring{\mathbb{D}}_n$, $s \otimes \sqrt{d^n(x, y)_\tau} \in \mathcal{H}_\tau^\hbar$ and $s' \otimes \sqrt{d^n(x, y)_{\tau'}} \in \mathcal{H}_{\tau'}^\hbar$,
\begin{align*}
	\langle s \otimes \sqrt{d^n(x, y)_\tau}, s' \otimes \sqrt{d^n(x, y)_{\tau'}} \rangle_{\operatorname{BKS}} := \langle s, s' \rangle_{L^2} \langle \sqrt{d^n(x, y)_\tau}, \sqrt{d^n(x, y)_{\tau'}} \rangle.
\end{align*}
It is shown in \cite{BMN2010} 
that the BKS pairing can be extended continuously to the boundary of $\mathbb{D}_n$.
In later sections, we shall focus on the quantum Hilbert space $\mathcal{H}^\hbar := \mathcal{H}_{-I_n}^\hbar$ and $\check{\mathcal{H}}^\hbar := \mathcal{H}_{I_n}^\hbar$ for the two specific points $-I_n, I_n$ on the boundary of $\mathbb{D}_n$. By the work of \'Sniatycki \cite{S1977}, $\mathcal{P}^T$-polarized sections of $L^{\otimes k}$ are generated by a compactly supported distributional section with support on each $\hbar$-level Bohr-Sommerfeld fibre of the fibration $\check{\mu}: X \to \check{T}$, and similarly for $\mathcal{P}^{\check{T}}$-polarized sections of $L^{\otimes k}$. We postpone the explicit description of Bohr-Sommerfeld bases of $\mathcal{H}^\hbar$ and $\check{\mathcal{H}}^\hbar$ to the next section, after the introduction to Weil-Brezin transform.

\section{Fibrewise Fourier Transform and Weil-Brezin Transform}
\label{Section 5}
For the purpose of technical calculations in Section \ref{Section 6}, in this section we shall describe transformations of sections of $L^\otimes k$ and functions on $X$ into other spaces.\par
Recall that smooth sections of $L^{\otimes k}$ is identified as elements in $\mathcal{C}_{\operatorname{qper}, k}^\infty(\mathbb{R}^{2n}, \mathbb{C})$. For $k = 0$, (\ref{Equation 4.3}) still makes sense and hence $\mathcal{C}_{\operatorname{qper}, 0}^\infty(\mathbb{R}^{2n}, \mathbb{C})$ is well defined. We simply write $\mathcal{C}_{\operatorname{per}}^\infty(\mathbb{R}^{2n}, \mathbb{C})$ in place of $\mathcal{C}_{\operatorname{qper}, 0}^\infty(\mathbb{R}^{2n}, \mathbb{C})$. In particular, we identify $\mathcal{C}^\infty(X, \mathbb{C})$ as $\mathcal{C}_{\operatorname{per}}^\infty(\mathbb{R}^{2n}, \mathbb{C})$ (for any $m \in \mathbb{N}$, we can indeed define $\mathcal{C}_{\operatorname{per}}^\infty(\mathbb{R}^m, \mathbb{C})$ similarly and have an identification $\mathcal{C}^\infty(\mathbb{R}^m/\mathbb{Z}^m, \mathbb{C}) \cong \mathcal{C}_{\operatorname{per}}^\infty(\mathbb{R}^m, \mathbb{C})$). Then the action of smooth functions on $X$ on smooth sections of $L^{\otimes k}$ is identified as the usual multiplication of smooth functions on $\mathbb{R}^{2n}$:
\begin{align*}
	\mathcal{C}_{\operatorname{per}}^\infty(\mathbb{R}^{2n}, \mathbb{C}) \times \mathcal{C}_{\operatorname{qper}, k}^\infty(\mathbb{R}^{2n}, \mathbb{C}) \to \mathcal{C}_{\operatorname{qper}, k}^\infty(\mathbb{R}^{2n}, \mathbb{C}).
\end{align*}
It would be useful to further transform $\mathcal{C}_{\operatorname{per}}^\infty(\mathbb{R}^{2n}, \mathbb{C})$ and $\mathcal{C}_{\operatorname{qper}, k}^\infty(\mathbb{R}^{2n}, \mathbb{C})$ into other spaces via the fibrewise Fourier transform and the Weil-Brezin transform.

\subsection{Fibrewise Fourier transform}
\label{Subsection 5.1}
\quad\par
In this subsection, we shall perform the fibrewise Fourier transform of smooth functions on $X$ with respect to the smooth fibration $\check{\mu}: X \to \check{T}$ and recollect some standard results of the decay rate of Fourier coefficients in classical Fourier analysis, that is necessary for the norm estimations in Section \ref{Section 6}.

\begin{definition}
	Let $f \in \mathcal{C}^\infty(X, \mathbb{C})$. For $m \in \mathbb{Z}^n$, the $m$\emph{th fibrewise Fourier coefficient} of $f$, denoted by $\widehat{f}_m$, is defined as the periodic function $\widehat{f}_m \in \mathcal{C}_{\operatorname{per}}^\infty(\mathbb{R}^n, \mathbb{C})$ given by
	\begin{equation}
	\widehat{f}_m(y) := \int_{[0, 1]^n} f(x, y) e^{-2\pi \sqrt{-1} m \cdot x} d^nx, \quad \text{ for any } y \in \mathbb{R}^n.
	\end{equation}
	The \emph{fibrewise Fourier transform} of $f$, denoted by $\widehat{f}$, is the defined as the element in
	\begin{align*}
	\prod_{m \in \mathbb{Z}^n} \mathcal{C}_{\operatorname{per}}^\infty(\mathbb{R}^n, \mathbb{C})
	\end{align*}
	given by $m \mapsto \widehat{f}_m$.
\end{definition}
The inverse transform of the fibrewise Fourier transform is
\begin{align*}
f(x, y) = \sum_{m \in \mathbb{Z}^n} \widehat{f}_m(y) e^{2\pi \sqrt{-1} m \cdot x}, \quad \text{for any } x, y \in \mathbb{R}^n.
\end{align*}
\par
For all $m \in \mathbb{Z}^n$, define
\begin{equation}
\label{Equation 5.2}
N_m = \prod_{i \in \{1, ..., n\}: m_i \neq 0} \left( \frac{1}{2\pi \sqrt{-1} m_i} \right).
\end{equation}

In the following lemma, the first part is a standard result, appearing in \cite{SS2002} for instance; the same kind of arguments for the second part appear in \cite{NS1991}.

\begin{lemma}
	\label{Lemma 5.2}
	Let $f \in \mathcal{C}^\infty(\mathbb{R}^n / \mathbb{Z}^n, \mathbb{C}) \cong \mathcal{C}_{\operatorname{per}}^\infty(\mathbb{R}^n, \mathbb{C})$ and $\widehat{f}: \mathbb{Z}^n \to \mathbb{C}$ be its Fourier transform, i.e.
	\begin{align*}
	\widehat{f}_m = \int_{[0, 1]^n} f(x) e^{-2\pi \sqrt{-1} m \cdot x} d^n x, \quad m \in \mathbb{Z}^n.
	\end{align*}
	\begin{enumerate}
		\item Let $r \in \mathbb{N}$ and $C > 0$ be such that for all subset $S$ of $\{1, ..., n\}$ and $x \in \mathbb{R}^n$,
		\begin{align*}
			\left\lvert \left( \prod_{i \in S} \frac{\partial^r}{\partial (x^i)^r} \right) f (x) \right\rvert \leq C.
		\end{align*}
		Then for all $m \in \mathbb{Z}^n$, $\lvert \widehat{f}_m \rvert \leq C \lvert N_m^r \rvert$, where $N_m$ is defined as in (\ref{Equation 5.2}).
		\item Let $r \in \mathbb{N}$. There exists $C_r > 0$ such that for all $k \in \mathbb{N}$,
		\begin{align*}
			\left\lvert \widehat{f}_0 - \frac{1}{k^n} \sum_{[m] \in \mathbb{Z}_k^n} f\left( \frac{m}{k} \right) \right\rvert \leq \frac{C_r}{k^r},
		\end{align*}
		where $\mathbb{Z}_k^n = \mathbb{Z}^n / k\mathbb{Z}^n$ and $[m]$ is the equivalence class of $m \in \mathbb{Z}^n$ in $\mathbb{Z}_k^n$.
	\end{enumerate}
\end{lemma}
\begin{proof}
	\begin{enumerate}
		\item 
		Fix $m \in \mathbb{Z}^n$. Define the differential operator
		\begin{align*}
		D_m = \prod_{i \in \{1, ..., n\}: m_i \neq 0} \left( \frac{\partial}{\partial x^i} \right).
		\end{align*}
		By applying integration by parts iteratively,
		\begin{align*}
		\widehat{f}_m = \int_{[0, 1]^n} f(x) e^{-2\pi \sqrt{-1} m \cdot x} d^n x = N_m^r \int_{[0, 1]^n} (D_m^r f)(x) e^{-2\pi \sqrt{-1} m \cdot x} d^n x,
		\end{align*}
		whence we have $\lvert \widehat{f}_m \rvert \leq \lvert N_m^r \rvert \cdot \int_{[0, 1]^n} \lvert D_m^r f (x) \rvert d^n x \leq C \lvert N_m^r \rvert$.
		\item We partition $\mathbb{Z}^n$ into the disjoint union $\mathbb{Z}^n = \bigsqcup_{i=0}^n Z_i$, where for all $i \in \{0, ..., n\}$, $Z_i$ is the set of $p \in \mathbb{Z}^n$ such that the number of non-zero $p_j$'s ($j \in \{1, ..., n\}$) is $i$. Clearly, $Z_0$ is the singleton set containing the zero element in $\mathbb{Z}^n$. It suffices to consider $r \in \mathbb{N}$ with $r \geq 2$. Since $\mathbb{R}^n / \mathbb{Z}^n$ is compact and $f$ is periodic, there exists $C_r > 0$ be such that for all subset $S$ of $\{1, ..., n\}$ and $x \in \mathbb{R}^n$, $\left\lvert \left( \prod_{i \in S} \frac{\partial^r}{\partial (x^i)^r} \right) f (x) \right\rvert \leq C_r$. Therefore, for all $p \in \mathbb{Z}^n$, $\lvert \widehat{f}_p \rvert \leq C_r \lvert N_p^r \rvert$.\par
		Fix $k \in \mathbb{N}$. We have
		\begin{align*}
			\frac{1}{k^n} \sum_{[m] \in \mathbb{Z}_k^n} f\left( \frac{m}{k} \right) - \widehat{f}_0 = \sum_{p \in \mathbb{Z}^n \backslash Z_0} \widehat{f}_p \frac{1}{k^n} \sum_{[m] \in \mathbb{Z}_k^n} e^{2\pi \sqrt{-1} k^{-1} p \cdot m} = \sum_{p \in \mathbb{Z}^n \backslash Z_0} \widehat{f}_{kp}.
		\end{align*}
		Denote by $\zeta(z)$ the Riemann zeta function. Note that $\sum_{m \in \mathbb{Z} \backslash \{0\}} \tfrac{1}{(2\pi k \lvert m \rvert)^r} = \tfrac{2\zeta(r)}{(2\pi k)^r}$. Then
		\begin{align*}
		\left\lvert \sum_{p \in \mathbb{Z}^n \backslash Z_0} \widehat{f}_{kp} \right\rvert \leq \sum_{i=1}^n \sum_{p \in Z_i} \left\lvert \widehat{f}_{kp} \right\rvert \leq C_r \sum_{i=1}^n \sum_{p \in Z_i} \lvert N_{kp}^r \rvert = \sum_{i=1}^n \frac{C_r}{k^{ri}} \binom{n}{i} \left( \frac{\zeta(r)}{(2\pi)^r} \right)^i.
		\end{align*}
	\end{enumerate}
\end{proof}

\begin{lemma}
	\label{Lemma 5.3}
	Let $f \in \mathcal{C}^\infty(X, \mathbb{C})$, $I$ be any multi-index and $r \in \mathbb{N}$. For $m \in \mathbb{Z}^n$, define $N_m$ as in (\ref{Equation 5.2}). Then the following statements hold.
	\begin{enumerate}
		\item There exists $C_{f, I, r} > 0$ such that for all $m \in \mathbb{Z}^n$ and $y \in \mathbb{R}^n$, $\left\lvert m^I \widehat{f}_m(y) \right\vert \leq C_{f, I, r} \lvert N_m^r \rvert$.
		\item There exists $C'_{f, I, r} > 0$ such that for all $m \in \mathbb{Z}^n$ and $y \in \mathbb{R}^n$, $\left\lvert \widehat{f}_m^{(I)}(y) \right\vert \leq C'_{f, I, r} \lvert N_m^r \rvert$, where $\widehat{f}_m^{(I)}(y) = \tfrac{\partial^{\lvert I \rvert}}{\partial y^I} \widehat{f}_m(y)$.
	\end{enumerate}
\end{lemma}
\begin{proof}
	As $X$ is compact and $f$ is periodic, there exists $C_{f, r} > 0$ such that for all subset $S$ of $\{1, ..., n\}$ and $x, y \in \mathbb{R}^n$,
	\begin{align*}
	\left\lvert \left( \prod_{i \in S} \frac{\partial^r}{\partial (x^i)^r} \right) f(x, y) \right\rvert \leq C_{f, r}.
	\end{align*}
	Then by Lemma \ref{Lemma 5.2}, for all $m \in \mathbb{Z}^n$ and $y \in \mathbb{R}^n$, $\lvert \widehat{f}_m(y) \rvert \leq C_{f, r} \lvert N_m^r \rvert$. Now consider any multi-index $I$. Note that $(2\pi\sqrt{-1})^{-\lvert I \rvert} \tfrac{\partial^{\lvert I \rvert} f}{\partial x^I}$ and $\frac{\partial^{\lvert I \rvert} f}{\partial y^I}$ are compactly supported smooth functions on $X$ as well and for each $m \in \mathbb{Z}^n$, the $m$th fibrewise Fourier coefficient of $(2\pi\sqrt{-1})^{-\lvert I \rvert} \tfrac{\partial^{\lvert I \rvert} f}{\partial x^I}$ is $m^I \hat{f}_m$ while that of $\tfrac{\partial^{\lvert I \rvert} f}{\partial y^I}$ is $\widehat{f}_m^{(I)}$. Repeating the above arguments, we are done.
\end{proof}

\subsection{Weil-Brezin transform}
\quad\par
In this subsection, we shall introduce the Weil-Brezin transform \cite{F1989} of sections of $L^{\otimes k}$, which is parallel to the fibrewise Fourier transform of functions on $X$ in subsection \ref{Subsection 5.1}. Define $\mathbb{Z}_k = \mathbb{Z}/k\mathbb{Z}$ to be the cyclic group of order $k$. For $m \in \mathbb{Z}^n$, denote by $[m]$ its equivalence class $m + k\mathbb{Z}^n$ in $\mathbb{Z}_k^n$. We have another way to describe sections of $L^{\otimes k}$. If $s$ is a smooth section of $L^{\otimes k}$, then the function $\tilde{s}$ on $\mathbb{R}^{2n}$ defined by
\begin{equation}
\tilde{s}(x, y) := e^{k \pi \sqrt{-1} x \cdot y} s(x, y)
\end{equation}
is periodic in $x$ but not periodic in $y$: for all $\lambda \in \mathbb{Z}^n$,
\begin{align*}
\tilde{s}(x + \lambda, y) = & \tilde{s}(x, y),\\
\tilde{s}(x, y + \lambda) = & 
e^{k \pi \sqrt{-1} x \cdot (y + 2\lambda)} \tilde{s}(x, y).
\end{align*}
These observations lead to an isomorphism given by the Weil-Brezin expansion.

\begin{definition}
	Let $s \in \Gamma^\infty(X, L^{\otimes k}) \cong \mathcal{C}_{\operatorname{qper}, k}^\infty(\mathbb{R}^{2n}, \mathbb{C})$. For $[m] \in \mathbb{Z}_k^n$, The $[m]$\emph{th Weil-Brezin coefficient} of $s$, denoted by $(s)_{[m]}$, is defined as the Schwartz function $(s)_{[m]} \in \mathcal{S}(\mathbb{R}^{2n})$ given by
	\begin{align*}
		(s)_{[m]}(y) := \int_{[0, 1]^n} \tilde{s}(x, y + \hbar m) e^{-2\pi\sqrt{-1} m \cdot x} d^nx, \quad \text{for any } y \in \mathbb{R}^n.
	\end{align*}
	The \emph{Weil-Brezin transform} of $s$, denoted by $(s)$, is defined as
	\begin{align*}
		(s) \in \prod_{[m] \in \mathbb{Z}_k^n} \mathcal{S}(\mathbb{R}^n)
	\end{align*}
	given by $[m] \mapsto (s)_{[m]}$.
\end{definition}

The Weil-Brezin transform is an isomorphism between topological vector spaces $\Gamma^\infty(X, L^{\otimes k})$ and $\prod_{[m] \in \mathbb{Z}_k^n} \mathcal{S}(\mathbb{R}^n)$. The inverse formula is as follows. For $s \in \Gamma^\infty(X, L^{\otimes k})$, we have
\begin{equation}
	s(x, y) = e^{-k\pi \sqrt{-1} y \cdot x} \sum_{[m] \in \mathbb{Z}_k^n} \sum_{p \in [m]} (s)_{[m]}(y - p) e^{2 \pi \sqrt{-1} p \cdot x}.
\end{equation}
Thus, the Weil-Brezin transform extends to an isomorphism $\Gamma^{-\infty}(X, L^{\otimes k}) \cong \prod_{[m] \in \mathbb{Z}_k^n} \mathcal{S}'(\mathbb{R}^n)$. This map in turn restricts to a unitary isomorphism:
\begin{align*}
\Gamma_{L^2}(X, L^{\otimes k}) \cong \prod_{[m] \in \mathbb{Z}_k^n} L^2(\mathbb{R}^n).
\end{align*}
In particular, for all $a, b \in \Gamma_{L^2}(X, L^{\otimes k})$,
\begin{equation}
\langle a, b \rangle_{L^2} = \sum_{[m] \in \mathbb{Z}_k^n} \int_{\mathbb{R}^n} (a)_{[m]}(y) \overline{(b)_{[m]}(y)} d^ny.
\end{equation}
For $f \in \mathcal{C}^\infty(X ,\mathbb{C})$ and $s \in \Gamma^\infty(X, L^{\otimes k})$, the Weil-Brezin transform of $fs$ can be expressed in terms of the fibrewise Fourier transform $\widehat{f}$ of $f$ and the Weil-Brezin transform $(s)$ of $s$: for $[m] \in \mathbb{Z}_k^n$ and $y \in \mathbb{R}^{2n}$,
\begin{equation}
\label{Equation 5.6}
(fs)_{[m]}(y)
= \sum_{p \in \mathbb{Z}^n} \widehat{f}_p(y + \hbar m) (s)_{[m - p]}(y + \hbar p).
\end{equation}
Our next task is to describe the weak equation of covariant constancy via the Weil-Brezin transform. Fix $\tau \in \mathbb{D}_n$. We define the first order differential operators $\Xi_\tau^\hbar: \mathcal{S}'(\mathbb{R}^n) \to \mathcal{S}'(\mathbb{R}^n)^n$ by
\begin{align*}
\Xi_\tau^\hbar = (I_n + \tau) \frac{\partial}{\partial y} + 2k\pi (I_n - \tau) y.
\end{align*}
From Lemma 3.6 in \cite{BMN2010}, the Hilbert space of quantum states can be expressed as
\begin{equation}
\mathcal{H}_\tau^\hbar = \left\{ s \in \Gamma^{-\infty}(X, L^{\otimes k}): (s) \in \prod_{[m] \in \mathbb{Z}_k^n} \ker \Xi_\tau^\hbar \right\} \otimes \mathbb{C} \sqrt{d^n(x, y)_\tau}.
\end{equation}
Consider the case when $\tau \in \mathring{\mathbb{D}}_n$. For $m \in \mathbb{Z}_k^n$, define $\vartheta_{\tau, \hbar}^m \in \Gamma^\infty(X, L^{\otimes k})$ via its Weil-Brezin coefficients:
\begin{equation}
( \vartheta_{\tau, \hbar}^m )_{m'}(y) = \delta_{m, m'} \det(I_n + \tau)^{-\frac{1}{2}}  e^{k\pi\sqrt{-1} y \cdot \Omega(\tau)y}, \quad m' \in \mathbb{Z}_k^n.
\end{equation}
Furthermore, we define
\begin{equation}
\sigma_{\tau, \hbar}^m := 2^{\frac{n}{4}} k^{\frac{n}{2}} \vartheta_{\tau, \hbar}^m \sqrt{d^n(x, y)_\tau}.
\end{equation}
By Theorem 3.10 in \cite{BMN2010}, for each $m \in \mathbb{Z}_k^n$, the assignment $\mathring{\mathbb{D}}_n \ni \tau \mapsto \vartheta_{\tau, \hbar}^m \in \Gamma^\infty(X, L^{\otimes k})$ extends continuously to a map $\mathbb{D}_n \to \Gamma^{-\infty}(X, L^{\otimes k})$, and for any point $\tau$ on the boundary of $\mathbb{D}_n$, the element $\sigma_{\tau, \hbar}^m := 2^{\frac{n}{4}} k^{\frac{n}{2}} \vartheta_{\tau, \hbar}^m \sqrt{d^n(x, y)_\tau}$ lies in $\mathcal{H}_\tau^\hbar$. We are particularly interested in the case when $\tau = I_n$ and when $\tau = -I_n$. We simplify the notations by rewriting $\vartheta_{-I_n, \hbar}^m$ as $\vartheta_\hbar^m$, $\sigma_{-I_n, \hbar}^m$ as $\sigma_\hbar^m$ and $\sqrt{d^n(x, y)_{-I_n}}$ as $\rho$, (resp. $\vartheta_{I_n, \hbar}^m$ as $\check{\vartheta}_\hbar^m$, $\sigma_{I_n, \hbar}^m$ as $\check{\sigma}_\hbar^m$ and $\sqrt{d^n(x, y)_{I_n}}$ as $\check{\rho}$).

\begin{itemize}
	\item Consider the case when $\tau = -I_n$, corresponding to polarization $\mathcal{P}^T$. Then $\Xi_{-I_n}^\hbar = 4k\pi y$. Hence, $\ker \Xi_{-I_n}^\hbar$ is spanned by the Dirac delta distribution supported at $y = 0$. For $m, m' \in \mathbb{Z}_k^n$, we see that for all $f \in \mathcal{S}(\mathbb{R}^n)$,
	\begin{equation}
	\label{Equation 5.10}
	(\vartheta_\hbar^m)_{m'}(f) = \delta_{m, m'} (2k)^{-\frac{n}{2}} f(0).
	\end{equation}
	We can also verify that for all $\tau \in \Gamma^\infty(X, L^{\otimes k})$,
	\begin{equation}
	\langle \vartheta_\hbar^m, \tau \rangle = (2k)^{-\frac{n}{2}} \int_{[0, 1]^n} e^{\pi\sqrt{-1} \tilde{m} \cdot x} \overline{\tau(x, \hbar \tilde{m})} d^nx,
	\end{equation}
	where $\tilde{m} \in \mathbb{Z}^n$ is any representative of $m$.
	\item Consider the case when $\tau = I_n$, corresponding to polarization $\mathcal{P}^{\check{T}}$. Then $\Xi_{I_n}^\hbar = 2 \partial_y$. Thus, solutions are constant functions. For $m, m' \in \mathbb{Z}_k^n$, we see that for all $y \in \mathbb{R}^n$,
	\begin{equation}
	\label{Equation 5.12}
	(\check{\vartheta}_\hbar^m)_{m'}(y) = \delta_{m, m'} 2^{-\frac{n}{2}}.
	\end{equation}
	We can also verify that for all $\tau \in \Gamma^\infty(X, L^{\otimes k})$,
	\begin{equation}
		\langle \check{\vartheta}_\hbar^m, \tau \rangle = 2^{-\frac{n}{2}} \int_{[0, 1]^n} e^{-\pi\sqrt{-1} \tilde{m} \cdot y} \overline{\tau(\hbar \tilde{m}, y)} d^ny,
	\end{equation}
	where $\tilde{m} \in \mathbb{Z}^n$ is any representative of $m$.
\end{itemize}
We can evaluate by direct calculation that for $m, m' \in \mathbb{Z}_k^n$,
\begin{equation}
	\label{Equation 5.14}
	\left\langle \sigma_\hbar^m, \check{\sigma}_\hbar^{m'} \right\rangle_{\operatorname{BKS}} = \delta_{m, m'}.
\end{equation}

\section{Berezin-Toeplitz Quantization on Symplectic Tori}
\label{Section 6}
This is the main section of this paper. We shall define a Toeplitz-type operator
\begin{align*}
Q^\hbar: \mathcal{C}^\infty(X, \mathbb{C}) \times \mathcal{H}^\hbar \to \mathcal{H}^\hbar, \quad (f, s) \mapsto Q_f^\hbar s
\end{align*}
acting on the quantum Hilbert space $\mathcal{H}^\hbar$ in the real polarization $\mathcal{P}^T$ and provide proofs of our main theorems.

\subsection{Construction of Toeplitz operators for a pair of transversal real polarizations}
\quad\par
The key issue for constructing a Toeplitz-type operator in our case is to overcome the difficulty that quantum states in real polarizations $\mathcal{P}^T, \mathcal{P}^{\check{T}}$ are distributional sections, and to find the analogue to the projection $\Pi^\hbar$ appeared in (\ref{Equation 2.1}), which can be represented as
\begin{align}
\label{Equation 6.2}
\Pi^\hbar s = \sum_{i \in I} \langle s, e_\hbar^i \rangle_{L^2} e_\hbar^i,
\end{align}
where $\{ e_\hbar^i \}_{i \in I}$ is a unitary basis in $H^0(X, L^{\otimes k})$, for $X$ being compact K\"ahler. In our case when $X = T \times \check{T}$, instead of a basis of a single quantum Hilbert space, we have bases $\{ \sigma_\hbar^m \}_{m \in \mathbb{Z}_k^n}$ and $\{ \check{\sigma}_\hbar^m \}_{m \in \mathbb{Z}_k^n}$ of $\mathcal{H}^\hbar$ and $\check{\mathcal{H}}^\hbar$ respectively, satisfying the relation (\ref{Equation 5.14}) so that they can be regarded as dual bases of each other (see Section \ref{Section 5}). Motivated by (\ref{Equation 6.2}), we define the (distributional) kernel
\begin{equation}
K^\hbar = \sum_{m \in \mathbb{Z}_k^n} \check{\sigma}_\hbar^m \otimes \sigma_\hbar^m \in \check{\mathcal{H}}^\hbar \otimes_\mathbb{C} \mathcal{H}^\hbar
\end{equation}
and hope to define $\Pi^\hbar$ by the expression $\Pi^\hbar s = \sum_{m \in \mathbb{Z}_k^n} \langle s, \check{\sigma}_\hbar^m \rangle \sigma_\hbar^m$ for some pairing $\langle \quad, \quad \rangle$ between distributional sections of $L^{\otimes k}$ and $\check{\mathcal{H}}^\hbar$, extending the BKS pairing $\langle \quad, \quad \rangle_{\operatorname{BKS}}: \mathcal{H}^\hbar \times \check{\mathcal{H}}^\hbar \to \mathbb{C}$.\par
In general we do not have a pairing on the space of distributional sections of $L^{\otimes k}$. For our purpose, however, it suffices to pair the spaces $\Gamma_{\operatorname{WF}^\hbar}^{-\infty}(X, L^{\otimes k})$ and $\Gamma_{\check{\operatorname{WF}}^\hbar}^{-\infty}(X, L^{\otimes k})$ of distributional sections of $L^{\otimes k}$ with wavefront sets lying in the unions $\operatorname{WF}^\hbar$ and $\check{\operatorname{WF}}^\hbar$ of wavefront sets of elements in $\mathcal{H}^\hbar$ and $\check{\mathcal{H}}^\hbar$ respectively. Note that the unions $\operatorname{BS}^\hbar$ and $\check{\operatorname{BS}}^\hbar$ of level-$\hbar$ Bohr Sommerfeld fibres for the polarizations $\mathcal{P}^T$ and $\mathcal{P}^{\check{T}}$ respectively are $n$-dimensional submanifolds of $X$ intersecting transversally. One can show that $\operatorname{WF}^\hbar$ and $\check{\operatorname{WF}}^\hbar$ are conormal bundles of $\operatorname{BS}^\hbar$ and $\check{\operatorname{BS}}^\hbar$ respectively, with zero sections removed. 
Then $\operatorname{WF}^\hbar$, $\check{\operatorname{WF}}^\hbar$ satisfy a variant of the \emph{H\"ormander's criterion} \cite{H2003} (see Subsection 4.14 in \cite{MU2008} and Proposition 6.10 in \cite{C} as well), namely $\operatorname{WF}^\hbar \cap \check{\operatorname{WF}}^\hbar = \emptyset$. It guarantees that we have a pairing
\begin{equation}
\label{Equation 6.4}
\langle \quad, \quad \rangle: \Gamma_{\operatorname{WF}^\hbar}^{-\infty}(X, L^{\otimes k}) \times \Gamma_{\check{\operatorname{WF}}^\hbar}^{-\infty}(X, L^{\otimes k}) \to \mathbb{C}
\end{equation}
which coincides with $\langle \quad, \quad \rangle_{L^2}$ when restricted to $L^2$ sections. By in-cooperating with the pairing on half-forms, we have a pairing
\begin{equation*}
\langle \quad, \quad \rangle: \Gamma_{\operatorname{WF}^\hbar}^{-\infty}(X, L^{\otimes k}) \otimes \mathbb{C} \rho \times \Gamma_{\check{\operatorname{WF}}^\hbar}^{-\infty}(X, L^{\otimes k}) \otimes \mathbb{C} \check{\rho} \to \mathbb{C}
\end{equation*}
extending the BKS pairing $\langle \quad, \quad \rangle_{\operatorname{BKS}}: \mathcal{H}^\hbar \times \check{\mathcal{H}}^\hbar \to \mathbb{C}$. We can now well define a $\mathbb{C}$-linear map
\begin{equation}
\Pi^\hbar: \Gamma_{\operatorname{WF}^\hbar}^{-\infty}(X, L^{\otimes k}) \otimes \mathbb{C} \rho \to \mathcal{H}^\hbar, \quad s \mapsto \Pi^\hbar s := (\langle \quad, \quad \rangle \otimes \operatorname{Id})(s \otimes K^\hbar).
\end{equation}

It is evident that $\Pi^\hbar$ is a projection. As multiplication by a smooth function never enlarges the wavefront set of a distributional section, the following definition makes sense.

\begin{definition}
	The \emph{Toeplitz operator for the pair of polarizations} $(\mathcal{P}^{\check{T}}, \mathcal{P}^T)$ is the map
	\begin{align*}
	Q^\hbar: \mathcal{C}^\infty(X, \mathbb{C}) \times \mathcal{H}^\hbar \to \mathcal{H}^\hbar, \quad (f, s) \mapsto Q_f^\hbar s := \Pi^\hbar(fs).
	\end{align*}
\end{definition}

The following proposition expresses this Toeplitz operator explicitly in terms of a basis.

\begin{proposition}
	Suppose $f \in \mathcal{C}^\infty(X, \mathbb{C})$. If $s = \sum_{[m] \in \mathbb{Z}_k^n} s_{[m]} \sigma_\hbar^{[m]} \in \mathcal{H}^\hbar$, then
	\begin{equation}
	\label{Equation 6.6}
	Q_f^\hbar s = \sum_{[m] \in \mathbb{Z}_k^n} \left( \sum_{m' \in \mathbb{Z}^n} \widehat{f}_{m - m'}(\hbar m') s_{[m']} \right) \sigma_\hbar^{[m]}.
	\end{equation}
\end{proposition}
\begin{proof}
	We take advantage of the Weil-Brezin transform and observe that the pairing (\ref{Equation 6.4}) can be expressed as the sum of componentwise natural (sesquilinear) pairing between compactly supported distributions and smooth functions on $\mathbb{R}^n$. Fix $[m], [m'] \in \mathbb{Z}_k^n$. By (\ref{Equation 5.6}) and (\ref{Equation 5.12}), for all $y \in \mathbb{R}^n$,
	\begin{align*}
		(\overline{f} \check{\vartheta}_\hbar^{[m]})_{[m']} (y) = & 2^{-\frac{n}{2}} \sum_{p \in [m']} \overline{\widehat{f}_{-(m'-p)}(y + \hbar m')}.
	\end{align*}
	Then by (\ref{Equation 5.10}),
	\begin{align*}
		\left\langle f s_{[m']} \vartheta_\hbar^{[m']}, \check{\vartheta}_\hbar^{[m]} \right\rangle = s_{[m']} \left\langle \vartheta_\hbar^{[m']}, \overline{f} \check{\vartheta}_\hbar^{[m]} \right\rangle = 2^{-n} k^{-\frac{n}{2}} s_{[m']} \sum_{p \in [m']} \widehat{f}_{m - p}(\hbar m').
	\end{align*}
	Therefore, by (\ref{Equation 4.10}),
	\begin{align*}
		\left\langle fs, \check{\sigma}_\hbar^{[m]} \right\rangle = 2^{-\frac{n}{2}} k^{\frac{n}{2}} \left\langle \rho, \check{\rho} \right\rangle \sum_{[m'] \in \mathbb{Z}_k^n} s_{[m']} \sum_{p \in [m']} \widehat{f}_{m - p}(\hbar m') = \sum_{m' \in \mathbb{Z}^n} \widehat{f}_{m - m'}(\hbar m') s_{[m']}.
	\end{align*}
\end{proof}

We provide some insights into formula (\ref{Equation 6.6}). As a toy model, consider $n = 1$, i.e. $X = \mathbb{R}^2 / \mathbb{Z}^2$ with periodic coordinates $(x, y)$. We take the universal covering $\tilde{X} = \mathbb{R}^2$ of $X$, on which $x, y$ are globally defined functions. When we take the real polarization spanned by $\tfrac{\partial}{\partial x}$, we expect to have the following assignments:
\begin{equation*}
x \mapsto \frac{\hbar}{2\pi \sqrt{-1}} \frac{\partial}{\partial y}; \quad y \mapsto y \cdot,
\end{equation*}
sending coordinates on $\tilde{X}$ to operators on the quantum Hilbert space $L^2(\mathbb{R})$. For $X$, there are two amendments. First, the quantum Hilbert space $\mathcal{H}^\hbar$ is now isomorphic to $l^2(\mathbb{Z}_k)$, where $\mathbb{Z}_k$ is embedded in $\mathbb{S}^1$ via $y + k\mathbb{Z} \mapsto e^{2\pi\sqrt{-1}\hbar y}$. Second, a correct analogue to $x, y$ are the globally defined functions $e^{2\pi \sqrt{-1}x}, e^{2\pi \sqrt{-1}y}$ on $X$ respectively. Under the identification $\mathcal{H}^\hbar \cong l^2(\mathbb{Z}_k)$, $Q^\hbar$ takes the following assignments
\begin{equation*}
e^{2\pi \sqrt{-1}x} \mapsto S_\hbar; \quad e^{2\pi \sqrt{-1}y} \mapsto e^{2\pi \sqrt{-1}y} \cdot,
\end{equation*}
where $S_\hbar$ is the shift operator on $l^2(\mathbb{Z}_k)$ by $\hbar$. Quantization on $X$ is therefore a discrete version of that on $\tilde{X}$.\par
We also see from (\ref{Equation 6.6}) that the Toeplitz operator $Q^\hbar$ for the pair of polarizations $(\mathcal{P}^{\check{T}}, \mathcal{P}^T)$ gives rise to a representation of a quantum torus (c.f. Appendix \ref{Appendix B}).

\subsection{Norm estimations}
\quad\par
Now we start the discussion on norm estimations. To facilitate the estimate of the operator norm $\lVert A \rVert$ for any operator $A$ on $\mathcal{H}^\hbar$ with respect to the inner product on $\mathcal{H}^\hbar$, we resort to the expedient of defining the norms
\begin{align*}
	\lVert s \rVert_1 = \sum_{m \in \mathbb{Z}_k^n} \lvert s_m \rvert \quad \text{and} \quad \lVert s \rVert_\infty = \sup_{m \in \mathbb{Z}_k^n} \lvert s_m \rvert,
\end{align*}
for any $s = \sum_{m \in \mathbb{Z}_k^n} s_m \sigma_\hbar^m \in \mathcal{H}^\hbar$, and denote by $\lVert A \rVert_1$ and $\lVert A \rVert_\infty$ the operator norm of $A$ with respect to $\lVert \quad \rVert_1$ and $\lVert \quad \rVert_\infty$ respectively. The key reason behind the introduction of these two ad hoc norms is the following variant of H\"older inequality (one might consult Chapter 5 in \cite{HJ2013}). 

\begin{lemma}
	\label{Lemma 6.5}
	Let $r \in \mathbb{N}$. For $1 \leq p \leq \infty$, let $\lVert \quad \rVert_p$ be the operator norm with respect to the $l^p$ norm on $\mathbb{C}^r$. Suppose $1 \leq p, q \leq \infty$ are H\"older conjugates, i.e. $\tfrac{1}{p} + \tfrac{1}{q} = 1$ (by convention, when $p = 1$, $q = \infty$). Then for any $r \times r$ complex matrix $A$,
	\begin{equation*}
		\lVert A \rVert_2 \leq \sqrt{\lVert A \rVert_p \lVert A \rVert_q}.
	\end{equation*}
\end{lemma}
\begin{proof}
	Recall that the spectral radius $\rho(A)$ of an $r \times r$ complex matrix $A$ is always a lower bound of operator norms of $A$. In particular, $\lVert A \rVert_2 = \sqrt{\rho(A^*A)}$, where $A^*$ is the Hermitian transpose of $A$. An additional observation is that $\lVert A^* \rVert_p = \lVert A \rVert_q$.
	From this, it yields
	\begin{equation*}
	\lVert A \rVert_2 = \sqrt{\rho(A^*A)} \leq \sqrt{\lVert A^*A \rVert_p} \leq \sqrt{\lVert A \rVert_p \lVert A^* \rVert_p} = \sqrt{\lVert A \rVert_p \lVert A \rVert_q}.
	\end{equation*}
\end{proof}
	
\begin{proposition}
	For all $f \in \mathcal{C}^\infty(X, \mathbb{C})$, there exists $K(f) > 0$ such that
	\begin{align*}
		\left\lVert Q_f^\hbar \right\rVert \leq K(f)
	\end{align*}
	 for all $k \in \mathbb{N}$, where $\hbar = \tfrac{1}{k}$.
\end{proposition}
\begin{proof}
	By Lemma \ref{Lemma 5.3}, there exists $K(f) > 0$ such that for all map $\mathbb{Z}^n \to \mathbb{R}^n$, $m \mapsto y_m$, $\sum_{m \in \mathbb{Z}^n} \lvert \widehat{f}_m(y_m) \rvert \leq K(f)$. Let $k \in \mathbb{N}$ and $\hbar = \tfrac{1}{k}$. For any $s \in \mathcal{H}^\hbar$, writing $s = \sum_{[m] \in \mathbb{Z}_k^n} s_{[m]} \sigma_\hbar^{[m]}$,
	\begin{align*}
	\left\lVert Q_f^\hbar s \right\lVert_1 = & \sum_{[m] \in \mathbb{Z}_k^n} \left\lvert \sum_{m' \in \mathbb{Z}^n} \widehat{f}_{m-m'}(\hbar m') s_{[m']} \right\rvert
	\leq \sum_{[m'] \in \mathbb{Z}_k^n} \lvert s_{[m']} \rvert \sum_{m \in \mathbb{Z}^n} \lvert \widehat{f}_m (\hbar m') \lvert \leq K(f) \lVert s \rVert_1.\\
	\left\lVert Q_f^\hbar s \right\lVert_\infty = & \sup_{[m] \in \mathbb{Z}_k^n} \left\lvert \sum_{m' \in \mathbb{Z}^n} \widehat{f}_{m-m'}(\hbar m') s_{[m']} \right\rvert \leq \lVert s \rVert_\infty \sup_{[m] \in \mathbb{Z}_k^n} \sum_{m' \in \mathbb{Z}^n} \left\lvert \widehat{f}_{m-m'}(\hbar m') \right\rvert \leq K(f) \lVert s \rVert_\infty.
	\end{align*}
	Thus, $\lVert Q_f^\hbar \rVert_1 \leq K(f)$ and $\lVert Q_f^\hbar \rVert_\infty \leq K(f)$. By Lemma \ref{Lemma 6.5}, $\lVert Q_f^\hbar \rVert \leq \sqrt{\lVert Q_f^\hbar \rVert_1 \lVert Q_f^\hbar \rVert_\infty} \leq K(f)$.
\end{proof}

Therefore, the operator norm $\lVert Q_f^\hbar \rVert$ of $Q_f^\hbar$ has a uniform bound, independent of $\hbar$.\par
Before providing the proof of Theorem \ref{Theorem 1.1}, we give a heuristic argument for $X = \mathbb{R}^2/\mathbb{Z}^2$, demonstrating the idea. Suppose $p, q, m \in \mathbb{Z}$. Consider smooth functions
\begin{align*}
f(x, y) = \widehat{f}_p(y) e^{2\pi \sqrt{-1} px} \quad \text{and} \quad g(x, y) = \widehat{g}_q(y) e^{2\pi\sqrt{-1} qx}
\end{align*}
on $X$. We see from (\ref{Equation 3.3}) and (\ref{Equation 3.4}) that
\begin{align*}
(f \star g)(x, y) = \sum_{i=0}^\infty \frac{( \hbar q )^i}{i!} \widehat{f}_p^{(i)}(y) \widehat{g}_q(y) e^{2\pi \sqrt{-1}(p+q)x} = \widehat{f}_p(y + \hbar q) \widehat{g}_q(y) e^{2\pi \sqrt{-1}(p+q)x}
\end{align*}
formally. 
Then by (\ref{Equation 6.6}), formally,
\begin{align*}
Q_{f \star g}^\hbar \sigma_\hbar^{[m]} = \widehat{f}_p(\hbar(m +  q)) \widehat{g}_q(\hbar m) \sigma_\hbar^{[m+p+q]} = Q_f^\hbar \left( \widehat{g}_q(\hbar m) \sigma_\hbar^{[m+q]} \right) = (Q_f^\hbar \circ Q_g^\hbar) \sigma_\hbar^{[m]}.
\end{align*}
The above phenomenon is described by Theorem \ref{Theorem 1.1}. Recall that for $N \in \mathbb{N} \cup \{0\}$, $f, g \in \mathcal{C}^\infty(X, \mathbb{C})$ and $\hbar \in \mathbb{R}^+$, $f \star_N g = \sum_{i=0}^N \hbar^i C_i(f, g)$.

\begin{theorem}
	\label{Theorem 6.7}
	($=$ Theorem \ref{Theorem 1.1}) For all $f, g \in \mathcal{C}^\infty(X, \mathbb{C})$ and $N \in \mathbb{N} \cup \{0\}$, there exists $K = K_N(f, g) > 0$ such that
	\begin{align*}
	\left\lVert Q_f^\hbar \circ Q_g^\hbar - Q_{f \star_N g}^\hbar \right\rVert & \leq K \hbar^{N+1}
	\end{align*}
	for all $k \in \mathbb{N}$, where $\hbar = \tfrac{1}{k}$.
\end{theorem}


\begin{proof} 
	Fix $f, g \in \mathcal{C}^\infty(X, \mathbb{C}) \cong \mathcal{C}_{\operatorname{per}}^\infty(\mathbb{R}^{2n}, \mathbb{C})$ and $N \in \mathbb{N} \cup \{0\}$. We shall construct a bound $K$ by estimating the decay rate of the fibrewise Fourier coefficients of functions $f, g$ and their derivatives. For $m \in \mathbb{Z}^n$, define $N_m$ as in (\ref{Equation 5.2}). By Lemma \ref{Lemma 5.3}, for all multi-index $I$, there exist $C_{f, I}, C_{g, I} > 0$ such that for all $m \in \mathbb{Z}^n$ and $y \in \mathbb{R}^n$,
	\begin{align*}
	\left\lvert \widehat{f}_m^{(I)}(y) \right\vert \leq C_{f, I} \lvert N_m^2 \rvert \quad \text{and} \quad \left\lvert m^I \widehat{g}_m(y) \right\vert \leq C_{g, I} \lvert N_m^2 \rvert.
	\end{align*}
	We see that $S := \sum_{m \in \mathbb{Z}^n} \lvert N_m^2 \rvert < +\infty$. Define
	\begin{align*}
	K = S^2 \sum_{\lvert I \rvert = N+1} \frac{C_{f, I} C_{g, I}}{I!} > 0.
	\end{align*}
	Consider any $k \in \mathbb{N}$ and $\hbar = \tfrac{1}{k}$. Define the error term $E_N^\hbar := Q_f^\hbar \circ Q_g^\hbar - Q_{f \star_N g}^\hbar$. Fix a quantum state $s = \sum_{[m] \in \mathbb{Z}_k^n} s_{[m]} \sigma_\hbar^{[m]} \in \mathcal{H}^\hbar$. One the one hand, by (\ref{Equation 6.6}), we see that
	\begin{align*}
	Q_f^\hbar Q_g^\hbar s = & \sum_{[m] \in \mathbb{Z}_k^n} \left( \sum_{p, r \in \mathbb{Z}^n} \widehat{f}_{m-p}(\hbar p) \widehat{g}_{p-r}(\hbar r) s_{[r]} \right) \sigma_\hbar^{[m]}\\
	= & \sum_{[m] \in \mathbb{Z}_k^n} \left( \sum_{r, q \in \mathbb{Z}^n} \widehat{f}_{m-r-q}(\hbar (r+q)) \widehat{g}_q(\hbar r) s_{[r]} \right) \sigma_\hbar^{[m]}.
	\end{align*}
	On the other hand, for any $i \in \mathbb{N} \cup \{0\}$ and $r \in \mathbb{Z}^n$, the $r$th fibrewise Fourier coefficient of $C_i(f, g)$ is
	\begin{align*}
	\sum_{\lvert I \rvert = i} \frac{1}{I!} \sum_{q \in \mathbb{Z}^n} q^I \cdot \widehat{f}_{r-q}^{(I)}(y) \widehat{g}_q(y),
	\end{align*}
	and hence
	\begin{align*}
	Q_{C_i(f, g)}^\hbar s = & \sum_{[m] \in \mathbb{Z}_k^n} \left( \sum_{r, q \in \mathbb{Z}^n} \sum_{\lvert I \rvert = i} \frac{q^I}{I!} \widehat{f}_{m-r-q}^{(I)}(\hbar r) \widehat{g}_q(\hbar r) s_{[r]} \right) \sigma_\hbar^{[m]}.
	\end{align*}
	For all $m, r, q \in \mathbb{Z}^n$, define the remainder term
	\begin{align*}
	R_{m, r, q}^N := \widehat{f}_{m-r-q}(\hbar (r + q)) - \sum_{\lvert I \rvert \leq N} \frac{(\hbar q)^I}{I!} \widehat{f}_{m-r-q}^{(I)}(\hbar r),
	\end{align*}
	which can be expressed in integral form by Taylor's Theorem:
	\begin{align*}
	R_{m, r, q}^N = (N+1) \sum_{\lvert I \rvert = N+1} \frac{\hbar^{N+1} q^I}{I!} \int_0^1 (1 - t)^N \widehat{f}_{m-r-q}^{(I)}(\hbar (r + t q)) dt,
	\end{align*}
	and we can furthermore see that
	\begin{align*}
	\lvert R_{m, r, q}^N \rvert \leq \hbar^{N+1} \lvert N_{m-r-q}^2 \rvert \sum_{\lvert I \rvert = N+1} \frac{C_{f, I} \lvert q^I \rvert}{I!}.
	\end{align*}
	Therefore, the error term becomes
	\begin{align*}
		E_N^\hbar(s) = \sum_{[m] \in \mathbb{Z}_k^n} \left( \sum_{r, q \in \mathbb{Z}^n} R_{m, r, q}^N \widehat{g}_q(\hbar r) s_{[r]} \right) \sigma_\hbar^{[m]}.
	\end{align*}
	Then we have
	\begin{align*}
	\lVert E_N^\hbar(s) \rVert_1 
	\leq & \hbar^{N+1} \sum_{[m] \in \mathbb{Z}_k^n} \sum_{r, q \in \mathbb{Z}^n} \lvert s_{[r]} \rvert \lvert N_{m-r-q}^2 \rvert \lvert N_q^2 \rvert \sum_{\lvert I \rvert = N+1} \frac{C_{f, I} C_{g, I}}{I!}
	= K \hbar^{N+1} \lVert s \rVert_1,\\
	\lVert E_N^\hbar(s) \rVert_\infty 
	\leq & \hbar^{N+1} \lVert s \rVert_\infty \sup_{[m] \in \mathbb{Z}_k^n} \sum_{q \in \mathbb{Z}^n} \sum_{r \in \mathbb{Z}^n} \lvert N_{m-r-q}^2 \rvert \lvert N_q^2 \rvert \sum_{\lvert I \rvert = N+1} \frac{C_{f, I} C_{g, I}}{I!}
	= K \hbar^{N+1} \lVert s \rVert_\infty.
	\end{align*}
	Thus, $\lVert E_N^\hbar \rVert_1 \leq K \hbar^{N+1}$ and $\lVert E_N^\hbar \rVert_\infty \leq K \hbar^{N+1}$. By Lemma \ref{Lemma 6.5}, $\lVert E_N^\hbar \rVert \leq K \hbar^{N+1}$.
\end{proof}

This verifies that as $\hbar \to 0^+$, $Q^\hbar$ gives an asymptotic representation of the star product $\star$:
\begin{align*}
	Q_f^\hbar \circ Q_g^\hbar \approx Q_{f \star g}^\hbar.
\end{align*}
\par
It is not hard to swap the roles of the two polorizations $\mathcal{P}^T, \mathcal{P}^{\check{T}}$ and construct similarly the \emph{Toeplitz operator for the pair} $(\mathcal{P}^T, \mathcal{P}^{\check{T}})$:
\begin{align*}
	\check{Q}^\hbar \times \check{\mathcal{H}}^\hbar \to \check{\mathcal{H}}^\hbar, \quad (f, s) \mapsto \check{Q}_f^\hbar s.
\end{align*}
One can verify that for all $f \in \mathcal{C}^\infty(X, \mathbb{C})$ and $s = \sum_{[m] \in \mathbb{Z}_k^n} s_{[m]} \check{\sigma}_\hbar^{[m]} \in \check{\mathcal{H}}^\hbar$,
\begin{equation}
\label{Equation 6.7}
\check{Q}_f^\hbar s = \sum_{[m] \in \mathbb{Z}_k^n} \left( \sum_{m' \in \mathbb{Z}^n} \widehat{f}_{m - m'}(\hbar m) s_{[m']} \right) \check{\sigma}_\hbar^{[m]}.
\end{equation}
For deformation quantization, we know from Proposition \ref{Proposition 3.3} that the two star products $\star, \check{\star}$ are related by an isomorphism $\mathcal{F}^{\operatorname{DQ}} := e^{-\hbar \Delta}$; for geometric quantization, we know from Equation (29) in \cite{BMN2010} that the quantum Hilbert spaces $\mathcal{H}^\hbar, \check{\mathcal{H}}^\hbar$ are related by the unitary BKS pairing maps
\begin{equation}
\label{Equation 6.8}
B^\hbar: \check{\mathcal{H}}^\hbar \to \mathcal{H}^\hbar, \quad B^\hbar(\check{\sigma}_\hbar^m) = \sigma_\hbar^m, \quad \text{for any } m \in \mathbb{Z}_k^n,
\end{equation}
which is essentially the discrete Fourier transform, and the formulae (\ref{Equation 6.6}) and (\ref{Equation 6.7}) match the shift theorem for discrete Fourier transform. Let $\mathcal{F}^{\operatorname{GQ}}: \operatorname{End}_\mathbb{C} \check{\mathcal{H}}^\hbar \to \operatorname{End}_\mathbb{C} \mathcal{H}^\hbar$ be the conjugation by $B^\hbar$.\par
We now give a heuristic argument of the commutativity of the diagram (\ref{Equation 1.2}) on the toy model $X = \mathbb{R}^2 / \mathbb{Z}^2$. In this case, $\Delta = \tfrac{\sqrt{-1}}{2\pi} \frac{\partial^2}{\partial x \partial y}$. Suppose $f \in \mathcal{C}^\infty(X, \mathbb{C})$, formally,
\begin{align*}
(\mathcal{F}^{\operatorname{DQ}} f)(x, y) = \sum_{m \in \mathbb{Z}} \sum_{i=0}^\infty \frac{(\hbar m)^i}{i!}\widehat{f}_m^{(i)}(y) e^{2\pi \sqrt{-1} mx} \approx \sum_{m \in \mathbb{Z}} \widehat{f}_m(y + \hbar m) e^{2\pi \sqrt{-1} mx},
\end{align*}
spotted again that for all $m \in \mathbb{Z}$, the term $\sum_{i=0}^\infty \frac{(\hbar m)^i}{i!}\widehat{f}_m^{(i)}(y)$ is the Taylor series of $\widehat{f}_m(y + \hbar m)$. Then for any $m' \in \mathbb{Z}$,
\begin{align*}
Q_{\mathcal{F}^{\operatorname{DQ}} f}^\hbar \sigma_\hbar^{[m']} \approx & \sum_{[m] \in \mathbb{Z}_k} \sum_{p \in [m']} \widehat{f}_{m-p}(\hbar p + \hbar (m-p)) \sigma_\hbar^{[m]}\\
= & \sum_{[m] \in \mathbb{Z}_k} \sum_{p \in [m']} \widehat{f}_{m-p}(\hbar m) \sigma_\hbar^{[m]} = B^\hbar \check{Q}_f^\hbar (B^\hbar)^{-1}(\sigma_\hbar^{[m']}).
\end{align*}
Now for all $N \in \mathbb{N} \cup \{0\}$, we define $\mathcal{F}_N^{\operatorname{DQ}} = \sum_{i=0}^N \frac{1}{i!} (-\hbar \Delta)^i$ to be the $N$th order truncation of $\mathcal{F}^{\operatorname{DQ}}$.

\begin{theorem}
	($=$ Theorem \ref{Theorem 1.2}) For all $f \in \mathcal{C}^\infty(X, \mathbb{C})$ and $N \in \mathbb{N} \cup \{0\}$, there exists $K > 0$ such that
	\begin{align*}
	\lVert (\mathcal{F}^{\operatorname{GQ}} \circ \check{Q}^\hbar)(f) - (Q^\hbar \circ \mathcal{F}_N^{\operatorname{DQ}})(f) \rVert \leq K \hbar^{N+1}
	\end{align*}
	for all $k \in \mathbb{N}$, where $\hbar = \tfrac{1}{k}$.
\end{theorem}
\begin{proof}
	Fix $f \in \mathcal{C}^\infty(X, \mathbb{C}) \cong \mathcal{C}_{\operatorname{per}}^\infty(\mathbb{R}^{2n}, \mathbb{C})$ and $N \in \mathbb{N} \cup \{0\}$. For $m \in \mathbb{Z}^n$, define $N_m$ as in (\ref{Equation 5.2}). By Lemma \ref{Lemma 5.3}, for all multi-index $I$, there exists $C_{f, I} > 0$ such that for all $m \in \mathbb{Z}^n$ and $y \in \mathbb{R}^n$,
	\begin{align*}
		\left\lvert m^I \widehat{f}_m(y) \right\vert \leq C_{f, I} \lvert N_m^2 \rvert.
	\end{align*}
	Note again that $S := \sum_{m \in \mathbb{Z}^n} \lvert N_m^2 \rvert < +\infty$. Define $K = S \sum_{\lvert I \rvert = N+1} \frac{C_{f, I}}{I!} > 0$. Consider any $k \in \mathbb{N}$. Let $\hbar = \tfrac{1}{k}$. Define the error term $E_N^\hbar = (\mathcal{F}^{\operatorname{GQ}} \circ \check{Q}^\hbar)(f) - (Q^\hbar \circ \mathcal{F}_N^{\operatorname{DQ}})(f)$. Fix a quantum state $s = \sum_{[m] \in \mathbb{Z}_k^n} s_{[m]} \sigma_\hbar^{[m]} \in \mathcal{H}^\hbar$. One the one hand,
	\begin{align*}
	(\mathcal{F}^{\operatorname{GQ}} \circ \check{Q}^\hbar)(f)(s) = \sum_{[m] \in \mathbb{Z}_k^n} \left( \sum_{m' \in \mathbb{Z}^n} \widehat{f}_{m-m'}(\hbar m) s_{[m']} \right) \sigma_\hbar^{[m]} = \sum_{[m], [m'] \in \mathbb{Z}_k^n} \sum_{r \in [m - m']} \widehat{f}_r(\hbar m) s_{[m']} \sigma_\hbar^{[m]}.
	\end{align*}
	On the other hand, for any $m \in \mathbb{Z}^n$, the $m$th fibrewise Fourier coefficient of $\mathcal{F}_N^{\operatorname{DQ}} f$ is
	\begin{align*}
	\sum_{\lvert I \rvert \leq N} \frac{\hbar^{\lvert I \rvert} m^I}{I!}\widehat{f}_m^{(I)}(y),
	\end{align*}
	and hence
	\begin{align*}
	(Q^\hbar \circ \mathcal{F}_N^{\operatorname{DQ}})(f)(s) = & \sum_{[m], [m'] \in \mathbb{Z}_k^n} \sum_{r \in [m - m']} \sum_{\lvert I \rvert \leq N} \frac{\hbar^{\lvert I \rvert} r^I}{I!}\widehat{f}_r^{(I)}(\hbar (m-r)) s_{[m']} \sigma_\hbar^{[m]}.
	\end{align*}
	For all $m, r \in \mathbb{Z}^n$, define the remainder term
	\begin{align*}
	R_{m, r}^N := \widehat{f}_r(\hbar m) - \sum_{\lvert I \rvert \leq N} \frac{\hbar^{\lvert I \rvert} r^I}{I!}\widehat{f}_r^{(I)}(\hbar (m-r)),
	\end{align*}
	which can be expressed in integral form by Taylor's Theorem:
	\begin{align*}
	R_{m, r}^N = (N+1) \sum_{\lvert I \rvert = N+1} \frac{\hbar^{N+1} r^I}{I!} \int_0^1 (1 - t)^N \widehat{f}_r^{(I)}(\hbar (m - (1 - t) r)) dt,
	\end{align*}
	and we can furthermore see that
	\begin{align*}
	\lvert R_{m, r}^N \rvert \leq \hbar^{N+1} \lvert N_r^2 \rvert \sum_{\lvert I \rvert = N+1} \frac{C_{f, I}}{I!}.
	\end{align*}
	Therefore, the error term becomes $E_N^\hbar(s) = \sum_{[m], [m'] \in \mathbb{Z}_k^n} \sum_{r \in [m-m']} R_{m, r}^N s_{[m']} \sigma_\hbar^{[m]}$. Then
	\begin{align*}
		\lVert E_N^\hbar (s) \rVert_1 
		\leq & \hbar^{N+1} \sum_{[m'] \in \mathbb{Z}_k^n} \left\lvert s_{[m']} \right\rvert \sum_{[m] \in \mathbb{Z}_k^n} \sum_{r \in [m-m']} \lvert N_r^2 \rvert \sum_{\lvert I \rvert = N+1} \frac{C_{f, I}}{I!} = K \hbar^{N+1} \lVert s \rVert_1,\\
		\lVert E_N^\hbar (s) \rVert_\infty 
		\leq & \hbar^{N+1} \lVert s \rVert_\infty \sup_{[m] \in \mathbb{Z}_k^n} \sum_{[m'] \in \mathbb{Z}_k^n} \sum_{r \in [m-m']} \lvert N_r^2 \rvert \sum_{\lvert I \rvert = N+1} \frac{C_{f, I}}{I!} = K \hbar^{N+1} \lVert s \rVert_\infty.
	\end{align*}
	Thus, $\lVert E_N^\hbar \rVert_1 \leq K \hbar^{N+1}$ and $\lVert E_N^\hbar \rVert_\infty \leq K \hbar^{N+1}$. By Lemma \ref{Lemma 6.5}, $\lVert E_N^\hbar \rVert \leq K \hbar^{N+1}$.
\end{proof}

Finally, we compare traces of Toeplitz operators for the pair $(\mathcal{P}^{\check{T}}, \mathcal{P}^T)$ with the trace defined for deformation quantization with separation of variables in $(\mathcal{P}^{\check{T}}, \mathcal{P}^T)$.

\begin{theorem}
	($=$ Theorem \ref{Theorem 1.3}) Let $\operatorname{Tr}$ be given as in (\ref{Equation 3.7}) and $f \in \mathcal{C}^\infty(X, \mathbb{C})$. Then
	\begin{equation*}
	\operatorname{tr}( Q_f^\hbar ) = \operatorname{Tr}(f) + \hbar^{-n} \operatorname{O}(\hbar^\infty).
	\end{equation*}
\end{theorem}
\begin{proof}
	It directly follows from (\ref{Equation 6.6}) that if $k \in \mathbb{N}$ and $\hbar = \tfrac{1}{k}$, then
	\begin{align*}
	\operatorname{tr}( Q_f^\hbar ) = \sum_{[m] \in \mathbb{Z}_k^n} \sum_{p \in \mathbb{Z}^n} \widehat{f}_{kp}(\hbar m).
	\end{align*}
	Applying Lemma \ref{Lemma 5.2} on $\widehat{f}_0$, we know that for all $r \in \mathbb{N}$, there exists $C_{r, 1} > 0$ such that for all $k \in \mathbb{N}$ with $\hbar = \tfrac{1}{k}$,
	\begin{align*}
	\left\lvert \sum_{[m] \in \mathbb{Z}_k^n} \widehat{f}_0(\hbar m) - \operatorname{Tr}(f) \right\rvert \leq C_{r, 1} \hbar^{r-n}.
	\end{align*}
	Hence, it suffices to prove the claim that for all sufficiently large $r \in \mathbb{N}$, there exists $C_{r, 2} > 0$ such that for all $k \in \mathbb{N}$ with $\hbar = \tfrac{1}{k}$,
	\begin{align*}
	\left\lvert \sum_{[m] \in \mathbb{Z}_k^n} \sum_{p \in \mathbb{Z}^n \backslash Z_0} \widehat{f}_{kp}(\hbar m) \right\rvert \leq C_{r, 2} \hbar^{r-n},
	\end{align*}
	where $Z_i$ is the set of $p \in \mathbb{Z}^n$ such that the number of non-zero $p_j$'s ($j \in \{1, ..., n\}$) is $i$ for each $i \in \{0, ..., n\}$. The idea of proof of this claim is more or less the same as that of Lemma \ref{Lemma 5.2}. We only need to notice that for $r \in \mathbb{N}$ with $r \geq 2$, we shall use Lemma \ref{Lemma 5.3} to show the existence of a uniform bound $C_{r, 2} > 0$ such that for all $p \in \mathbb{Z}^n$ and $y \in \mathbb{R}^n$, $\lvert \widehat{f}_p(y) \rvert \leq C_{r, 2} \lvert N_p^r \rvert$, where $N_p$ is defined as in (\ref{Equation 5.2}). Then similarly, for all $k \in \mathbb{N}$, letting $\hbar = \tfrac{1}{k}$,
	\begin{align*}
	\left\lvert \sum_{[m] \in \mathbb{Z}_k^n} \sum_{p \in \mathbb{Z}^n \backslash Z_0} \widehat{f}_{kp}(\hbar m) \right\rvert 
	\leq k^n C_{r, 2} \sum_{i=1}^n \sum_{p \in Z_i} \lvert N_{kp}^r \rvert = C_{r, 2} \sum_{i=1}^n \hbar^{ri-n} \binom{n}{i} \left( \frac{\zeta(r)}{(2\pi)^r} \right)^i,
	\end{align*}
	where $\zeta(r)$ is the Riemann zeta function.
\end{proof}

In other words, as $\hbar \to 0^+$, the asymptotic expansion of $\operatorname{tr}(Q_f^\hbar)$ is exactly $\operatorname{Tr}(f)$.

\appendix
\section{Compact Symplectic Manifolds with Transversal Real Polarizations} 
\label{Appendix A}
Since in this paper we explore the possibility of constructing Toeplitz-type operators in real polarizations, we would like to consider a compact connected symplectic manifold $(X, \omega)$ with two transversal (non-singular) real polarizations $\mathcal{P}^T, \mathcal{P}^{\check{T}}$ such that the leaf spaces $T, \check{T}$ for $\mathcal{P}^T, \mathcal{P}^{\check{T}}$ respectively are smooth manifolds (without boundary) and the quotient maps $\mu: X \to T$ and $\check{\mu}: X \to \check{T}$ are Lagrangian fibrations. \par
As $X$ is connected, so are $T, \check{T}$ and $B := T \times \check{T}$. 
We also see that $\pi := \mu \times \check{\mu}: X \to B$ is a local diffeomorphism. 
By compactness of $X$ and connectedness of $B$, $X$ is a finite covering space of $B$. 
In particular, for all $b = (t, \check{t}) \in B$, the fibre $F_b := \pi^{-1}(\{b\})$ is the intersection $\mu^{-1}(\{t\}) \cap \check{\mu}^{-1}(\{\check{t}\})$ of the fibres of $\mu, \check{\mu}$ over $t, \check{t}$ respectively and the cardinality of $F_b$ is independent of $b$.\par
We furthermore focus on the case when $F_b$ is a singleton set for some $b \in B$. Then $\pi: X \to B$ is a diffeomorphism and clearly we have Lagrangian sections of $\mu$ and $\check{\mu}$. Therefore, $T \overset{\mu}{\leftarrow} X \overset{\check{\mu}}{\rightarrow} \check{T}$ is a twin Lagrangian fibration of index $n$ in the sense of Definition 2 in \cite{LY2007}, and we can argue that $T, \check{T}$ are affine tori of the form $T = V/\Lambda$ and $\check{T} = V^*/\Lambda'$ for some $n$-dimensional real vector space $V$ and lattices $\Lambda$ of $V$ and $\Lambda'$ of $V^*$. Thus, $X = T \times \check{T} = (V \oplus V^*)/(\Lambda \oplus \Lambda')$. Note that in general, $\Lambda'$ may not be the dual lattice of $\Lambda$.

\section{Representation of Quantum Tori induced by Toeplitz Operators}
\label{Appendix B}
In this appendix, we shall recall the definition of a quantum torus and show that our construction of Toeplitz-type operators induces representations of certain quantum tori.

\begin{definition}
	Let $q = (q_{ij})$ be an $n \times n$ complex matrix such that for all $i, j \in \{ 1, ..., n \}$, $q_{ij} = 1$ if $i = j$ and $q_{ij}q_{ji} = 1$ if $i \neq j$. The \emph{quantum torus}, is the $\mathbb{C}$-algebra generated over $\mathbb{C}$ by the variables $u_1, ..., u_n$ and their inverses, subject to the relations
	\begin{align*}
	u_iu_j = q_{ij}u_ju_i, \quad \text{for any } i, j \in \{ 1, ..., n \}.
	\end{align*}
\end{definition}

\begin{corollary}
	\label{Corollary 6.4}
	For all $i \in \{1, ..., n\}$, define $U_i^\hbar = Q_{f_i}^\hbar$ and $V_i^\hbar = Q_{g_i}^\hbar$, where 
	\begin{align*}
	f_i(x, y) = e^{2\pi \sqrt{-1} x^i} \quad \text{and} \quad g_i(x, y) = e^{2\pi \sqrt{-1} y^i}.
	\end{align*}
	Then for all $i, j \in \{1, ..., n\}$,
	\begin{align*}
	U_i^\hbar \circ U_j^\hbar = U_j^\hbar \circ U_i^\hbar, \quad V_i^\hbar \circ V_j^\hbar = V_j^\hbar \circ V_i^\hbar, \quad \text{and} \quad U_i^\hbar \circ V_j^\hbar = e^{2\pi \sqrt{-1} \hbar \delta_{i, j}} V_j^\hbar \circ U_i^\hbar.
	\end{align*}
	In particular, this gives a representation of the quantum torus with the $2n \times 2n$ matrix $(q_{ij})$ in $\mathcal{H}^\hbar$, where $q_{ij} = e^{2\pi\sqrt{-1} \hbar \delta_{i, j-n}}$ if $1 \leq i \leq j \leq 2n$; and $q_{ij} = q_{ji}^{-1}$ if $1 \leq j < i \leq 2n$.
\end{corollary}
\begin{proof}
	For $i \in \{1, ..., n\}$, let $\mathbf{1}_i = (0, ..., 0, 1, 0, ..., 0) \in \mathbb{Z}^n$, where $1$ is in the $i$th entry. Fix a quantum state $s = \sum_{[m] \in \mathbb{Z}_k^n} s_{[m]} \sigma_\hbar^{[m]} \in \mathcal{H}^\hbar$. Fix $i, j \in \{1, ..., n\}$. By direct calculation using (\ref{Equation 6.6}),
	\begin{align*}
	U_i^\hbar U_j^\hbar s = & \sum_{[m] \in \mathbb{Z}_k^n} e^{2\pi \sqrt{-1} \hbar (m_i + m_j)} s_{[m]} \sigma_\hbar^{[m]} = U_j^\hbar U_i^\hbar s,\\
	V_i^\hbar V_j^\hbar s = & \sum_{[m] \in \mathbb{Z}_k^n} s_{[m - \mathbf{1}_i - \mathbf{1}_j]} \sigma_\hbar^{[m]} = V_j^\hbar V_i^\hbar s,\\
	V_j^\hbar U_i^\hbar s = & \sum_{[m] \in \mathbb{Z}_k^n} e^{2\pi \sqrt{-1}\hbar (m_i - \delta_{i, j})} s_{[m-\mathbf{1}_j]} \sigma_\hbar^{[m]}\\
	= & e^{-2\pi \sqrt{-1} \hbar \delta_{ij}} \sum_{[m] \in \mathbb{Z}_k^n} e^{2\pi \sqrt{-1} \hbar m_i} s_{[m - \mathbf{1}_j]} \sigma_\hbar^{[m]} = e^{-2\pi \sqrt{-1} \hbar \delta_{ij}} U_i^\hbar V_j^\hbar s.
	\end{align*}
\end{proof}




\keywords{}

\bibliographystyle{amsplain}
\bibliography{References}

\end{document}